\documentclass{article}

\usepackage{graphicx}
\usepackage{indentfirst}
\usepackage{amsmath,amsfonts,amsthm,amssymb}
\usepackage{mathrsfs}
\usepackage[all]{xy}
\usepackage{hyperref}

\def\co{\colon\thinspace}
\DeclareMathAlphabet{\mathsfsl}{OT1}{cmss}{m}{sl}
\newcommand{\tensor}[1]{\mathsfsl{#1}}
\newcommand{\spinc}{{\mathrm{Spin}^c}}
\newcommand{\relspin}{\underline{\mathrm{Spin}^c}}
\newcommand{\Bslash}{\backslash\hspace{-3pt}\backslash}

\newtheorem{thm}{Theorem}[section]
\newtheorem{lem}[thm]{Lemma}

\newtheorem{prop}[thm]{Proposition}

\newtheorem{ques}[thm]{Question}

\theoremstyle{definition}
\newtheorem{defn}[thm]{Definition}
\newtheorem{notn}[thm]{Notation}
\newtheorem{rem}[thm]{Remark}
\newtheorem{construction}[thm]{Construction}

\begin{document}

\title{Some applications of Gabai's internal hierarchy}

\author{{\Large Yi NI}\\{\normalsize Department of Mathematics, Caltech, MC 253-37}\\
{\normalsize 1200 E California Blvd, Pasadena, CA
91125}\\{\small\it Email\/:\quad\rm yini@caltech.edu}}

\date{}
\maketitle

\begin{abstract}
Any Haken $3$--manifold (possibly with boundary consisting of tori) can be transformed into a $\mathrm{surface}\times S^1$ by a series of splitting and regluing along incompressible surfaces. This fact was proved by Gabai as an application of his sutured manifold theory. The first half of this paper provides a few technical details in the proof. In the second half of this paper, some applications of Gabai's theorem to Heegaard Floer homology are given. We refine the known results about the Thurson norm and fibrations. We also give some classification results for Floer simple knots in manifolds with positive $b_1$.
\end{abstract}

\section{Introduction}

Sutured manifold theory was introduced by Gabai \cite{G1} in order to construct taut foliations. In recent years, this theory has led to a lot of discoveries in gauge theory and Floer homology. In these applications, there are typically two ways to use sutured manifold theory. One way is to use the existence of taut foliations on closed manifolds \cite{KMNorm,OSzGenus}, the other way is to define an invariant for sutured manifolds and to study the decomposition formula for the invariant \cite{NiFibred,Ju2,KMSuture}. However, sometimes we find it convenient to directly work with closed manifolds without referring to taut foliations. In \cite{GabaiHier}, as a byproduct of the sutured manifold theory, Gabai introduced an internal hierarchy for closed Haken $3$--manifolds (or manifolds with boundary consisting of tori). This theory turns out to be useful in Floer homology as an alternative approach to the applications mentioned above.

In this paper, we will give an exposition of Gabai's internal hierarchy, then discuss some of its applications to Heegaard Floer homology. First of all, let us state Gabai's theorem \cite{GabaiHier}.

\begin{thm}[Gabai]\label{thm:GabaiHier}
Let $M$ be a Haken $3$--manifold such that $\partial M$ is a possibly empty union of tori, then there exists a sequence $M=M_1, \dots, M_n$ such that $M_{i+1}$ is
obtained from $M_i$ by splitting and regluing along a connected incompressible surface
and $M_n$ is homeomorphic to $\text{surface}\times S^1$.
\end{thm}

Later, we will state and prove a more precise version of the above theorem in a special case, Theorem~\ref{thm:PreciseGabai}. This version immediately allows us to reprove the fact that Heegaard Floer homology detects the Thurston norm of a closed $3$--manifold \cite{OSzGenus}. This approach also allows us to refine the results about Thurston norm and fibrations by taking the homological action into consideration. See Section~\ref{Sect:Norm} for more detail.

The new results we will prove are about Floer simple knots.
Suppose that $K$ is a rationally null-homologous knot in $Y$,
Ozsv\'ath--Szab\'o \cite{OSzKnot,OSzRatSurg} and Rasmussen
\cite{RasThesis} showed that $K$ specifies a filtration on
$\widehat{CF}(Y)$. The homology of the associated graded chain complex is the knot Floer homology $\widehat{HFK}(Y,K)$.
From the construction of knot Floer homology, one sees that
$$\mathrm{rank}\:\widehat{HFK}(Y,K)\ge\mathrm{rank}\:\widehat{HF}(Y),$$
for any rationally null-homologous knot $K\subset Y$. When the
equality holds, we say that the knot has {\it simple} knot Floer
homology, or this knot is {\it Floer simple}.

Clearly, the unknot in $Y$ is always Floer simple. 
Sometimes there are nontrivial Floer simple knots. For example, 
the core of a solid torus in the genus--$1$ Heegaard splitting of a lens space is Floer simple. Moreover, if two knots $(Y_1,K_1)$ and $(Y_2,K_2)$ are 
Floer simple, then their connected sum $(Y_1\#Y_2,K_1\#K_2)$ is also Floer simple. In particular, $(Y_1\#Y_2,K_1)$ is Floer simple.

It is an interesting problem to determine all Floer simple knots.
For example, Hedden \cite{HedBerge} and Rasmussen \cite{RasBerge}
showed that if a knot $L\subset S^3$ admits an integral lens space
surgery, then the core of the surgery is a Floer simple knot in
the lens space. Hence the classification of Floer simple knots in
lens spaces could lead to a resolution of Berge's conjecture on
lens space surgery.

A deep theorem of Ozsv\'ath--Szab\'o
\cite[Theorem~1.2]{OSzGenus} implies that the only Floer simple
knot in $S^3$ is the unknot. 
The author \cite{NiBorro} classified Floer simple knots in $\#^nS^1\times S^2$: they are essentially the Borromean knots.

Our main new result is the following theorem.

\begin{thm}\label{thm:Nullhomotopic}
Suppose that $Y$ is a closed irreducible $3$--manifold with nonzero Thurston norm, and that $K\subset Y$ is a {\bf null-homotopic} knot. If $K$ is Floer simple, namely, $$\mathrm{rank}\:\widehat{HFK}(Y,K)=\mathrm{rank}\:\widehat{HF}(Y),$$
then $K$ is the unknot.
\end{thm}

The basic strategy of the proof is to use Theorem~\ref{thm:GabaiHier} to reduce the question to the case where $Y$ is a surface bundle over $S^1$. In this case we have the following theorem.

\begin{thm}\label{thm:BundleCase}
Suppose that $Y$ is a closed surface bundle over $S^1$ with fiber of genus $>1$, and that $K\subset Y$ is a rationally null-homologous knot. If $K$ is Floer simple, then $K$ is the unknot.
\end{thm}

We note that the condition in Theorem~\ref{thm:Nullhomotopic} that $K$ is null-homotopic seems not to be very essential. It is possible to replace it with a much weaker condition. In fact, we expect a negative answer to the following question.

\begin{ques}
Suppose that $Y$ is a closed irreducible $3$--manifold with nonzero Thurston norm. Is there a nontrivial rationally null-homologous Floer simple knot $K\subset Y$?
\end{ques}

\subsection{Comparison with previous works}

As we mentioned at the beginning, there have been a lot of approaches of using sutured manifold theory in Floer homology. 

In \cite{KMNorm,OSzGenus}, the starting point is the existence of taut foliations, which implies the existence of weakly semi-fillable contact structures by the work of Eliashberg and Thurston \cite{ET}. Then one can use deep results in contact and symplectic topology to get the desired conclusion. This approach can be used in quite general cases. For technical reasons, Ozsv\'ath and Szab\'o \cite{OSzGenus} only stated the result about Thurston norm for twisted Heegaard Floer homology. 
Using the Universal Coefficients Theorem, we can get the results for untwisted Heegaard Floer homology \cite{NiNormCos}. More technical issues appear in this approach if we want to consider the homological actions on the Floer homology like what we do in Section~\ref{Sect:Norm}.

In \cite{Ju1}, Juh\'asz defined an invariant for a class of sutured manifolds, following directly the construction of Heegaard Floer homology \cite{OSzAnn1}. A decomposition formula for this invariant was proved in \cite{Ju2}. Such a formula can be used to reprove the results about genus and fibered knots. The class of sutured manifolds studied in \cite{Ju1,Ju2} is a little bit special: every boundary component of the sutured manifold must contain at least one suture. We need to do more in order to make this approach applicable to closed $3$--manifolds. The first step would be defining the invariant for more general sutured manifolds. See Lekili \cite{Lekili} for some results in this direction.

There is a way to define sutured manifold invariants indirectly as in \cite[Proposition~2.9]{NiFibred} and \cite{KMSuture}. Using results like Proposition~\ref{prop:CutReglue}, one can construct a closed $3$--manifold and regard the sutured manifold invariant as a ``bottommost'' summand of the Heegaard Floer homology of a closed $3$--manifold. This approach has the advantage that it avoids the technical difficulties in defining the invariant directly. One can prove that the ``total'' invariant defined in this way is isomorphic to the ``total'' invariant defined in the previous approach \cite{Ju2,NiClosedFib,Lekili}. On the other hand, the relative Spin$^c$ structures on the sutured manifold are not seen in the closed manifold. So this approach is ``coarser'' than the previous one.

The approach taken in our paper is a variant of the approach taken by Kronheimer and Mrowka \cite{KMSuture}. The only difference here is the use of Gabai's internal hierarchy, which allows us to stay in the world of closed $3$--manifolds. For example, in order to deal with closed manifolds in \cite[Theorem~7.21]{KMSuture}, Kronheimer and Mrowka use the trick of doubling the corresponding sutured manifolds. It is shown that the Floer homology of the doubled manifold has a direct summand $\mathbb C$, so the Floer homology of the original manifold is nonzero by the excision theorem. Our approach here directly shows that the Floer homology has a direct summand $\mathbb Z$ or $\mathbb C$, see Section~\ref{Sect:Norm}.

\subsection{Outline of the paper}

This paper is organized as follows. In Section~\ref{Sect:Sut}, we give a review of sutured manifold theory. In Section~\ref{Sect:Sketch}, we state a precise version of Gabai's theorem and sketch the proof as in \cite{GabaiHier}. Sections~\ref{Sect:FindC} and \ref{Sect:DecCx} are devoted to the technical details in the proof. In Section~\ref{Sect:Norm}, we switch to the Heegaard Floer world. After recalling the results about Thurston norm and fibrations, we apply the hierarchy to refine these results. In Section~\ref{Sect:FSimple}, we prove our theorems about Floer simple knots.

\vspace{5pt}\noindent{\bf Acknowledgements.}\quad This work was started when the author visited Simons Center for Geometry and Physics. The author is grateful to Zolt\'an Szab\'o for asking the question which motivated this work, and to Mirela \c{C}iperiani and David Gabai for helpful conversations. The author thanks Liling Gu for pointing out a mistake in an earlier version of this paper, and the referee for the comments which helped to improve the exposition. The author was
partially supported by an AIM Five-Year Fellowship and NSF grant
numbers DMS-1021956 and DMS-1103976.


\section{Sutured manifold theory}\label{Sect:Sut}

In this section, we review some of the basic materials in sutured manifold theory.

We begin by fixing the notation.

\begin{notn}
Let $M$ be a manifold possibly with boundary. Let $\mathrm{int}(M)$ be the interior of $M$. Let $|M|$ be the number of components of $M$.
Suppose that $N$ is a submanifold of $M$. Let $\nu(N)$ be a closed tubular neighborhood of $N$ in $M$, and let $M\Bslash N=\overline{M\backslash \nu(N)}$. 
\end{notn}

\begin{defn}
A {\it sutured manifold} $(M,\gamma)$ is a compact oriented
3--manifold $M$ together with a set $\gamma\subset \partial M$ of
pairwise disjoint annuli $A(\gamma)$ and tori $T(\gamma)$. The
core of each component of $A(\gamma)$ is a {\it suture}, and the
set of sutures is denoted by $s(\gamma)$. We often omit $\gamma$ when it is understood from the context.

Every component of $R(\gamma)=\partial M-\mathrm{int}(\gamma)$ is
oriented. Define $R_+(\gamma)$ (or $R_-(\gamma)$) to be the union
of those components of $R(\gamma)$ whose normal vectors point out
of (or into) $M$. The orientations on $R(\gamma)$ must be coherent
with respect to $s(\gamma)$. Occasionally, we also denote $R(\gamma)$ by $R(M,\gamma)$ or $R(M)$ if there is no confusion.
\end{defn}

As an example, let $S$ be a compact oriented surface, $M=S\times I$, $\gamma=(\partial S)\times I$,
$R_-(\gamma)=S\times0, R_+(\gamma)=S\times 1$, then $(M,\gamma)$ is a sutured manifold. In this case we say that $(M,\gamma)$ is a {\it product sutured manifold}.

\begin{defn}
Let $(M,\gamma)$ be a sutured manifold, and $S$ a properly
embedded surface in M, such that no component of $\partial S$
bounds a disk in $R(\gamma)$ and no component of $S$ is a disk
with boundary in $R(\gamma)$. Suppose that for every component
$\lambda$ of $S\cap\gamma$, one of 1)--3) holds:

1) $\lambda$ is a properly embedded non-separating arc in
$\gamma$.

2) $\lambda$ is a simple closed curve in an annular component $A$
of $\gamma$ in the same homology class as $A\cap s(\gamma)$.

3) $\lambda$ is a homotopically nontrivial curve in a toral
component $T$ of $\gamma$, and if $\delta$ is another component of
$T\cap S$, then $\lambda$ and $\delta$ are coherently oriented.

Then $S$ is called a {\it decomposing surface}, and $S$ defines a
{\it sutured manifold decomposition}
$$(M,\gamma)\stackrel{S}{\rightsquigarrow}(M',\gamma'),$$
where $M'=M-\mathrm{int}(\nu(S))$ and
\begin{eqnarray*}
\gamma'\;\;&=&(\gamma\cap M')\cup \nu(S'_+\cap
R_-(\gamma))\cup
\nu(S'_-\cap R_+(\gamma)),\\
R_+(\gamma')&=&((R_+(\gamma)\cap M')\cup S'_+)-\mathrm{int}(\gamma'),\\
R_-(\gamma')&=&((R_-(\gamma)\cap M')\cup
S'_-)-\mathrm{int}(\gamma'),
\end{eqnarray*}
where $S'_+$ ($S'_-$) is that component of
$\partial\nu(S)\cap M'$ whose normal vector points out of
(into) $M'$.
\end{defn}

\begin{defn}
Let $S$ be a compact oriented surface with connected components
$S_1,\dots,S_n$. We define
$$\chi_-(S)=\sum_i\max\{0,-\chi(S_i)\}.$$
Let $M$ be a compact oriented 3--manifold, $A$ be a compact
codimension--0 submanifold of $\partial M$. Let $h\in H_2(M,A)$.
The {\it Thurston norm} $\chi_-(h)$ of $h$ is defined to be the minimal
value of $\chi_-(S)$, where $S$ runs over all the properly embedded
surfaces in $M$ with $\partial S\subset A$ and $[S]=h$.
\end{defn}

\begin{defn}
Suppose $M$ is a compact $3$--manifold, a properly embedded
surface $S\subset M$ is {\it taut} if $\chi_-(S)=\chi_-([S])$ in
$H_2(M,\partial S)$, no proper subsurface of $S$ is
null-homologous, and if any component of $S$ lies in a homology
class that is represented by an embedded sphere then this
component is a sphere.
\end{defn}

\begin{defn}
A sutured manifold $(M,\gamma)$ is {\it taut}, if $M$ is
irreducible and $R(\gamma)$ is  Thurston norm
minimizing in $H_2(M,\gamma)$.

Suppose $S$ is a decomposing surface in $(M,\gamma)$, $S$
decomposes $(M,\gamma)$ into $(M',\gamma')$. We say $S$ is {\it taut} if
$(M',\gamma')$ is taut. In this case we also say that the sutured manifold decomposition is {\it taut}.
\end{defn}

\begin{defn}
Suppose $C$ is a properly embedded curve in a compact surface $F$.
We say $C$ is {\it efficient} in $F$ if
$$|C\cap\delta|=|[C]\cdot[\delta]|,\quad\text{for each boundary component $\delta$ of $F$.}$$
\end{defn}

Let us recall a basic existence theorem for taut decompositions from Gabai \cite{G1,G2}.

\begin{thm}\label{thm:ExistTaut}
Suppose $(M,\gamma)$ is a taut sutured manifold. Let $\lambda\subset R(\gamma)$ be a set of pairwise disjoint simple essential curves in $R(\gamma)$ such that $\lambda$ is efficient. Suppose $[\lambda]\in H_1(R(\gamma),\partial R(\gamma))$ is nonzero and lies in the image of the map $$\partial\co H_2(M,\partial M)\to H_1(\partial M,\gamma)\cong H_1(R(\gamma),\partial R(\gamma)),$$ then there exists a taut surface $T\subset M$ such that $T\cap R(\gamma)=\lambda$. Moreover, $T$ can be chosen so that every component of $T$ intersects $R(\gamma)$.
\end{thm}
\begin{proof}
Let $y\in H_2(M,\partial M)$ satisfy that $\partial y=[\lambda]$. By 
\cite[Lemma~0.7]{G2}, there exists a taut surface $S\subset M$ such that $[S]=y$ and $(\partial S)\cap R(\gamma)$ is efficient in $R(\gamma)$. Clearly, $[(\partial S)\cap R(\gamma)]=[\lambda]\in H_1(R(\gamma),\partial R(\gamma))$. Using 
\cite[Lemma~0.6]{G2}, we can find a taut surface $T$ with $T\cap R(\gamma)=\lambda$. Throwing away the components of $T$ that do not intersect $R(\gamma)$, we still get a taut surface whose intersection with $R(\gamma)$ is $\lambda$. 
\end{proof}

\begin{lem}{\rm\cite[Lemma~0.4]{G2}}\label{lem:RevTaut}
Suppose $(M,\gamma)\stackrel{S}{\rightsquigarrow}(M',\gamma')$ is a sutured manifold decomposition. If $(M',\gamma')$ is taut, then $(M,\gamma)$ is also taut.
\end{lem}

\begin{defn}
A decomposing surface is called a {\it product disk}, if it is a
disk which intersects $s(\gamma)$ in exactly two points. A
decomposing surface is called a {\it product annulus}, if it is an
annulus with one boundary component in $R_+(\gamma)$, and the
other boundary component in $R_-(\gamma)$. The sutured manifold decomposition associated to a product disk or product annulus is called a {\it product decomposition}.
\end{defn}

By definition, a product annulus is always a decomposing surface, so none of its boundary components bounds a disk in $R(\gamma)$. This rules out some trivial cases.

\begin{lem}{\rm\cite[Lemmas~2.2,\:2.5]{GabaiFibred}}\label{lem:ProdIff}
Suppose $(M,\gamma)\stackrel{S}{\rightsquigarrow}(M',\gamma')$ is a product decomposition. Then $(M,\gamma)$ is taut if and only if $(M',\gamma')$ is taut.
\end{lem}

In \cite[Section~4]{G1}, Gabai defined a complexity $C(M,\gamma)$ for a sutured $(M,\gamma)$. This complexity has value in a totally ordered set, and measures how far $(M,\gamma)$ is from being a product sutured manifold. Namely, $C(M,\gamma)$ obtains its infimum if and only if $(M,\gamma)$ is a product sutured manifold. 

\begin{notn}
Suppose that $N$ is a submanifold of a sutured manifold $(M,\gamma)$ such that $N\cap \partial M\subset \partial N$. Let $\partial_+N=R_+(\gamma)\cap N, \partial_-N=R_-(\gamma)\cap N$. We give $\partial_+N$ the same orientation as $\partial N$, and $\partial_-N$ the opposition orientation of $\partial N$. Let $\partial_vN=\partial N\backslash\mathrm{int}(R(\gamma)\cap N)$.
\end{notn}

\begin{defn}(Compare \cite[Definition~4.10]{G1})
Let $(M,\gamma)$ be a taut sutured manifold, and let $S$ be a maximal set of pairwise disjoint and nonparallel product annuli and product disks. Let $\mathscr R(M,\gamma)$ be the sutured manifold obtained from $(M,\gamma)$ by decomposing along $S$ and throwing away product sutured manifold components. $\mathscr R(M,\gamma)$ is called the {\it reduced sutured manifold} of $(M,\gamma)$. 
\end{defn}

The homeomorphism type of $\mathscr R(M,\gamma)$ as a sutured manifold is well-defined. This fact may be proved using the JSJ theory\footnote{We will not prove this fact here, since we do not need it in our paper.}. However, $\mathscr R(M,\gamma)$ may not be unique as a submanifold of $(M,\gamma)$ up to isotopy\footnote{For example, let $M=S^1\times B$, where $B$ is an orientable surface with exactly two boundary components $C_+,C_-$. Let $\gamma=\emptyset$, $R_{\pm}(\gamma)=S^1\times C_{\pm}$. Let $a_1,a_2,\dots,a_n\subset B$ be a maximal set of disjoint non-parallel arcs connecting $C_0$ to $C_1$. Then $S^1\times(B\Bslash \cup a_i)$ is a reduced sutured manifold. Since there may be different isotopy type of $B\Bslash \cup a_i$ in $B$, the isotopy type of $\mathscr R(M,\gamma)$ may be different.}. In practice, when we talk about $\mathscr R(M,\gamma)$ as a submanifold of $M$, we just choose one submanifold satisfying the definition.

\begin{defn}
The {\it reduced complexity} $\overline{C}(M,\gamma)$ of $(M,\gamma)$ is defined to be the complexity of the reduced sutured manifold $\mathscr R(M,\gamma)$.
\end{defn}

For simplicity, we omit the elaborate definition of $C(M,\gamma)$ in this paper. The only thing we need to know is the following lemma.

\begin{lem}\label{lem:Complexity}
Suppose $(M,\gamma)\stackrel{S}{\rightsquigarrow}(M',\gamma')$ is a taut sutured manifold decomposition such that some component of $S$ is not boundary parallel. Suppose that $(M,\gamma)$ does not contain any product annulus, and any product disk in $(M,\gamma)$ is boundary parallel,
then $$\overline{C}(M',\gamma')<C(M,\gamma)=\overline{C}(M,\gamma).$$
\end{lem}
\begin{proof}
By Step 2 in the proof of \cite[Theorem~4.2]{G1}, there exists a commutative diagram
$$
\xymatrix{
(M,\gamma)\ar@{~>}[r]^-{S}\ar@{~>}[rd]^-{S_1} &(M',\gamma')\ar@{~>}[d]^-{F}\\
 &(M'',\gamma''),
}
$$
where $F$ is a union of product disks, and $(M'',\gamma'')$ is the sutured manifold obtained from $(M',\gamma')$ by decomposing along $F$. Moreover, by Step 3 in \cite[Theorem~4.2]{G1}, there is an inequality 
$$C(M'',\gamma'')<C(M,\gamma).$$
Since $F$ consists of product disks, $\overline{C}(M',\gamma')=\overline{C}(M'',\gamma'')\le C(M'',\gamma'')$. Hence our conclusion holds.
\end{proof}


\section{Sketch of Gabai's construction}\label{Sect:Sketch}

In this section, we recall Gabai's sketch of the proof of Theorem~\ref{thm:GabaiHier} in \cite{GabaiHier}. For simplicity, we will only consider closed oriented irreducible $3$--manifolds with nonzero Thurston norm. This case is sufficient for our applications. For our purpose, we will strengthen the statement of Theorem~\ref{thm:GabaiHier} by introducing a few concepts in Heegaard Floer homology. Recall that the Heegaard Floer homology of a closed three-manifold splits as a direct sum with respect to 
Spin$^c$ structures:
$$HF^+(Y)\cong\bigoplus_{\mathfrak s\in\spinc(Y)}HF^+(Y,\mathfrak s).$$

\begin{defn}\label{defn:Basic}
A cohomology class $\alpha\in H^2(Y)$ is a {\it (Heegaard Floer) basic class} if $\alpha=c_1(\mathfrak s)$ for some $\mathfrak s\in\spinc(Y)$ with $HF^+(Y,\mathfrak s)\ne0$.
\end{defn}

The Adjunction Inequality \cite{OSzAnn2} says that if $\alpha$ is a basic class, then $$|\langle\alpha,[G]\rangle|\le\chi_-(G)$$ for any closed surface $G\subset Y$.

\begin{defn}
Suppose that $h\in H_2(Y)$. Let $$\mathcal B_Y(h)=\left\{\alpha=c_1(\mathfrak s)\left|\mathfrak s\in\spinc(Y),HF^+(Y,\mathfrak s)\ne0,\langle \alpha,h\rangle=-\chi_-(h)\right.\right\}$$
be the set of {\it bottommost basic classes} on $Y$ with respect to $h$.
\end{defn}

The following crucial observation 1) was made in \cite{KMSuture}.

\begin{lem}\label{lem:h1h2}
1) Suppose that $h_1,h_2\in H_2(Y)$ satisfy that $$\chi_-(h_1+h_2)=\chi_-(h_1)+\chi_-(h_2),$$ then
$$\mathcal B_Y(h_1+h_2)=\mathcal B_Y(h_1)\cap \mathcal B_Y(h_2).$$ 

2)  Suppose that $h_1,h_2\in H_2(Y)$ satisfy that $$\chi_-(h_1+h_2)<\chi_-(h_1)+\chi_-(h_2),$$ then
$$\mathcal B_Y(h_1)\cap \mathcal B_Y(h_2)\cap \mathcal B_Y(h_1+h_2)=\emptyset.$$ 

3) Given any two elements $h_1,h_2\in H_2(Y)$, we have
$$\mathcal B_Y(mh_1+h_2)\subset\mathcal B_Y(h_1)$$
when $m$ is sufficiently large.
\end{lem}
\begin{proof}
1) 
Suppose that $\alpha\in\mathcal B_Y(h_1)\cap \mathcal B_Y(h_2)$, then 
\begin{equation}\label{eq:h_1+h_2}
\langle\alpha,h_1+h_2\rangle=\langle\alpha,h_1\rangle+\langle\alpha,h_2\rangle=-\chi_-(h_1)-\chi_-(h_2)=-\chi_-(h_1+h_2).
\end{equation}
So $\mathcal B_Y(h_1)\cap\mathcal B_Y(h_2)\subset\mathcal B_Y(h_1+h_2)$. 

For the other inclusion,
suppose that $\alpha\in\mathcal B_{Y}(h_1+h_2)$, then
$$
\langle \alpha,h_1+h_2\rangle=-\chi_-(h_1+h_2),
$$
which implies that
$$\langle \alpha,h_1\rangle+\langle \alpha,h_2\rangle=-\chi_-(h_1)-\chi_-(h_2).
$$
On the other hand, as $\alpha$ is a basic class, by the Adjunction Inequality we have
$$\langle \alpha,h_i\rangle\ge-\chi_-(h_i).$$
So the equality must hold. In particular, $\alpha\in\mathcal B_Y(h_1)\cap \mathcal B_Y(h_2)$.

2) Suppose that $\alpha\in\mathcal B_Y(h_1)\cap \mathcal B_Y(h_2)$. Now the last equality in Equation (\ref{eq:h_1+h_2}) becomes ``$<$'',
so $\alpha\notin\mathcal B_Y(h_1+h_2)$.

3) By Thurston \cite[Theorem~2]{Th}, there exists a constant $C=C(h_1,h_2)$, such that $\chi_-(mh_1+h_2)=m\chi_-(h_1)+C$ when $m$ is sufficiently large. Our conclusion then follows from 1).
\end{proof}

\begin{defn}
Suppose $(Y,G),(Y',G')$ are two manifold-surface pairs such that there exists a series of manifold-surface pairs
$$(Y,G)=(Y_0,G_0), (Y_1,G_1), \dots, (Y_n,G_n)=(Y',G'),$$
such that $Y_{i+1}$ is obtained from $Y_i$ by cutting open $Y_i$ along $G_i$ and regluing by a homeomorphism of $G_i$, each $G_i$ is connected and taut in $Y_i$,
and
\begin{equation}\label{eq:SpincSub}
\mathcal B_{Y_{i+1}}([G_{i+1}])\subset\mathcal B_{Y_{i+1}}([G_i]),
\end{equation}
then we say that $(Y',G')$ is a {\it successor} of $(Y,G)$, denoted $(Y',G')\prec(Y,G)$. Here we also regard $G_i$ as a surface in $Y_{i+1}$, by abuse of notation.
\end{defn}

\begin{thm}\label{thm:PreciseGabai}
Suppose $Y$ is a closed irreducible $3$--manifold with $b_1>0$, and $G\subset Y$ is a connected taut surface  with $g(G)>1$. Then there exists a closed connected surface $F$ such that
$$(F\times S^1,F)\prec(Y,G).$$
\end{thm}

We sketch the proof here and give the technical details in the next two sections.

\begin{defn}
Suppose $G$ is a connected taut surface in $Y$. Let $(M,\gamma)$ be the sutured manifold obtained from $Y$ by cutting open $Y$ along $G$. (Here $\gamma=\emptyset$ since $Y$ is closed.) Then the {\it complexity} $C(Y,G)$ of the pair $(Y,G)$ is defined to be $\overline{C}(M,\gamma)$.
\end{defn}

Let $G$ be a connected taut surface in $Y$.

\noindent{\bf Step 1.} Let $M=Y\Bslash G$, $\partial M=G_+\cup G_-$. Find primitive homology classes $c_{\pm}\subset H_1(G_{\pm})$ such that $c_+=c_-$ in $H_1(M;\mathbb Q)$.

\noindent{\bf Step 2.} Find a homeomorphism $f\co G_+\to G_-$ with $f_*(c_+)=c_-$. Let $Y_1$ be the manifold obtained from $M$ by gluing $G_+$ to $G_-$ via $f$. Regard $G$ as a surface in $Y_1$. Find a closed surface $T\subset Y_1$ with $[T\cap G]=mc_+$ for some positive integer $m$.

\noindent{\bf Step 3.} Let $G_1\subset Y_1$ be a taut surface representing the homology class $[G]+[T]$. We can carefully make the choices in the previous two steps so that 
$$C(Y_1,G_1)<C(Y,G).$$ Moreover, we can prove $\mathcal B_{Y_{1}}([G_{1}])\subset\mathcal B_{Y_{1}}([G])$.

\noindent{\bf Step 4.} Repeat the above three steps to $(Y_1,G_1)$ to get a pair $(Y_2,G_2)$ with $C(Y_2,G_2)<C(Y_1,G_1)$ and $\mathcal B_{Y_{2}}([G_{2}])\subset\mathcal B_{Y_{2}}([G_1])$. Continue with this process until we get a $(Y_n,G_n)$ such that $Y_n$ fibers over the circle and $G_n$ is a fiber. Now it is easy to get $G_n\times S^1$ from $Y_n$.


\section{Finding homologous simple closed curves}\label{Sect:FindC}

The goal of this section is to prove Proposition~\ref{prop:HomoC}.

\begin{prop}\label{prop:HomoC}
Suppose that $M$ is a compact oriented $3$--manifold with boundary consisting of two homeomorphic connected surfaces $G_+,G_-$ with positive genus. Then there exist primitive homology classes $c_{\pm}\subset H_1(G_{\pm})$ such that $c_+=c_-$ in $H_1(M;\mathbb Q)$.
\end{prop}

Let us recall the following well-known fact, whose proof can be found in \cite{Mey,Sch}.

\begin{lem}
A homology class on a closed, oriented, connected surface is represented by
a non-separating simple closed curve if and only if it is primitive.
\end{lem}

The next lemma is also a standard fact.

\begin{lem}\label{lem:LagSpace}
Suppose $M$ is a compact oriented $3$--manifold with boundary. Let $\mathbb F$ be a field. Regard $H_1(\partial M;\mathbb F)$ as a symplectic space over $\mathbb F$ with respect to the intersection form $\omega$. Let
$$\mathcal K_{\mathbb F}=\ker \big(i_*\co H_1(\partial M;\mathbb F)\to H_1(M;\mathbb F)\big),$$
then $\mathcal K_{\mathbb F}$ is a Lagrangian subspace of $H_1(\partial M;\mathbb F)$.
\end{lem}
\begin{proof}
We will use $\mathbb F$ coefficients and suppress $\mathbb F$ in the proof. Without loss of generality, we may assume $M$ is connected.

The fact that $\mathcal K$ is an isotropic subspace of $H_1(\partial M)$ can be proved geometrically when $\mathbb F=\mathbb Q\text{ or }\mathbb Z/2\mathbb Z$. We present a more algebro-topological proof here. 

Using Poincar\'e duality, the exact sequence
$$
\xymatrix{
H_2(M,\partial M)\ar[r] &H_1(\partial M)\ar[r] &H_1(M)
}
$$
can be identified with
$$
\xymatrix{
H^1(M)\ar[r]^-{i^*}&H^1(\partial M)\ar[r]&H^2(M,\partial M).
}
$$
So $\mathcal K$ can be identified with the image of $i^*$. Suppose $i^*(\alpha),i^*(\beta)$ represent two elements in $\mathcal K$, then the intersection number of them is given by
\begin{equation}\label{eq:Cup}
\langle i^*(\alpha)\smile i^*(\beta),[\partial M]\rangle=\langle i^*(\alpha\smile \beta),[\partial M]\rangle.
\end{equation}
Consider the exact sequence
$$
\xymatrix{
H^2(M)\ar[r]^{i^*}&H^2(\partial M)\ar[r]^-{\delta}&H^3(M,\partial M)\cong\mathbb F.
}
$$
The map $\delta$ is given by the evaluation against $[\partial M]$, so the right hand side of (\ref{eq:Cup}) is zero. This shows that $\mathcal K$ is an isotropic subspace of $H_1(\partial M)$.

It remains to show that $\dim\mathcal K=\frac12\beta_1(\partial M)$.
Consider the long exact sequences
$$
0\to H_3(M,\partial M)\to H_2(\partial M)\to H_2(M)\to H_2(M,\partial M)\to\mathcal K\to0$$
$$
0\to H_1(\partial M)/\mathcal K\to H_1(M)\to H_1(M,\partial M)\to H_0(\partial M)\to H_0(M)\to0.
$$
By Poincar\'e duality we have $H_i(M,\partial M)\cong H^{3-i}(M)$ and $H_0(\partial M)\cong H^2(\partial M)$. The alternating sums of the dimensions of the items in the above long exact sequences are zero, which implies that
$$1-|\partial M|+\beta_2(M)-\beta_1(M)+\dim \mathcal K=0$$ and
$$(\beta_1(\partial M)-\dim\mathcal K)-\beta_1(M)+\beta_2(M)-|\partial M|+1=0.$$
So $\dim \mathcal K=\beta_1(\partial M)-\dim\mathcal K$.
\end{proof}

Let $$\mathcal H^{\pm}_{\mathbb F}=H_1(G_{\pm};\mathbb F),\quad\mathcal V^{\pm}_{\mathbb F}=\ker \big(\iota_*\co H_1(G_{\pm};\mathbb F)\to H_1(M;\mathbb F)\big).$$
If $\mathbb F=\mathbb Z/p\mathbb Z$ for some prime $p$, then let $\mathcal H^{\pm}_p=\mathcal H^{\pm}_{\mathbb F}$, $\mathcal V^{\pm}_p=\mathcal V^{\pm}_{\mathbb F}$.

\begin{lem}\label{lem:V+=V-}
Let $$\mathrm{Pr}_{\pm}\co\mathcal K_{\mathbb F}\to \mathcal H^{\pm}_{\mathbb F}$$
be the projection maps. Then
$$\mathcal V^{\pm}_{\mathbb F}=(\mathrm{im}\:\mathrm{Pr}_{\pm})^{\bot},$$
where the orthogonal complement is taken in $\mathcal H^{\pm}_{\mathbb F}$, with respect to the corresponding intersection form $\omega_{\pm}$.
Moreover, $$\dim\mathcal V^{+}_{\mathbb F}=\dim\mathcal V^{-}_{\mathbb F}.$$
\end{lem}
\begin{proof}
Suppose $a_+\in\mathcal V^+_{\mathbb F}$, $(b_+,b_-)\in \mathcal K_{\mathbb F}$, then $(a_+,0)\in\mathcal K_{\mathbb F}$, and
$$\omega_+(a_+,b_+)=\omega((a_+,0),(b_+,b_-))=0,$$
as $\mathcal K_{\mathbb F}$ is isotropic. So $a_+\in (\mathrm{im}\:\mathrm{Pr}_+)^{\bot}$, hence $$\mathcal V^{+}_{\mathbb F}\subset (\mathrm{im}\:\mathrm{Pr}_{+})^{\bot}.$$

On the other hand, if $a_+\in(\mathrm{im}\:\mathrm{Pr}_+)^{\bot}$, then for any $(b_+,b_-)\in \mathcal K_{\mathbb F}$ we have $\omega((a_+,0),(b_+,b_-))=0$, so $$(a_+,0)\in(\mathcal K_{\mathbb F})^{\bot}=\mathcal K_{\mathbb F}$$
since $\mathcal K_{\mathbb F}$ is Lagrangian. Thus $a_+\in \mathcal V^{+}_{\mathbb F}$. It follows that $$\mathcal V^{+}_{\mathbb F}\supset (\mathrm{im}\:\mathrm{Pr}_{+})^{\bot},$$ and hence the equality holds.

Similarly, $\mathcal V^{-}_{\mathbb F}= (\mathrm{im}\:\mathrm{Pr}_{-})^{\bot}$. 

Since 
$$\ker \mathrm{Pr}_+=0\oplus\mathcal V^-_{\mathbb F},$$
we have $$\dim\mathcal V^{+}_{\mathbb F}=\dim(\mathrm{coker}\:\mathrm{Pr}_+)=\dim(\ker \mathrm{Pr}_+)=\dim\mathcal V^{-}_{\mathbb F},$$
where the second equality uses the fact that $\dim\mathcal K_{\mathbb F}=\dim\mathcal H_{\mathbb F}^+=2g$.
\end{proof}

The proof of the following lemma is standard and left to the reader.

\begin{lem}\label{lem:SurgTorsFree}
Suppose $M$ is a compact oriented $3$--manifold. Then there exists a sequence of $3$--manifolds
$M=M_0,M_1,\dots,M_n$ for some $n\ge0$, such that
$M_{i+1}$ is obtained from $M_{i}$ by surgery on a rationally null-homologous knot $K_i$ in $M_{i}$, the maps in the following sequence $$H_1(M_{i};\mathbb Q)\to H_1(M_i, K_i;\mathbb Q)\cong H_1(M_{i+1},K_{i+1};\mathbb Q)\leftarrow H_1(M_{i+1};\mathbb Q)$$
are isomorphisms,
and $H_1(M_n;\mathbb Z)$ is torsion-free.
\end{lem}

\begin{rem}
By Lemma~\ref{lem:SurgTorsFree}, we can change $M$ to a manifold $M'$ with $H_1(M')$ torsion-free by a series of Dehn surgeries, and $M'$ has the same rational homology as $M$. We can work with $M'$ instead of $M$. 
\end{rem}

In this section, from now on we assume that $H_1(M;\mathbb Z)$ is torsion-free.

Let 
$$\mathcal H=H_1(\partial M),\quad\mathcal K=\ker \big(i_*\co H_1(\partial M)\to H_1(M)\big),$$
$$\mathcal H^{\pm}=H_1(G_{\pm}),\quad\mathcal V^{\pm}=\ker \big(\iota_*\co H_1(G_{\pm})\to H_1(M)\big).$$
Since $H_1(M)$ is torsion-free, $\mathcal K$ is a direct summand of $\mathcal H$, and $\mathcal V^{\pm}$ is a direct summand of $\mathcal H^{\pm}$.

\begin{lem}\label{lem:ExclPrime}
There exists a finite set of primes $\{p_1,\dots,p_n\}$ with the following properties. 

1) For any prime $q$ not in this set, if an element $c\in H_1(G_{\pm})$ maps to $0$ in $H_1(M;\mathbb Z/q\mathbb Z)$, then $c\in\mathcal V^{\pm}+q\mathcal H^{\pm}$. 

2) If an element $c\in \mathrm{Pr}_{\pm}(\mathcal K)$ is primitive in $\mathrm{Pr}_{\pm}(\mathcal K)$ and not divisible by any $p_i$ as an element in $\mathcal H^{\pm}$, then $c$ is also primitive in $\mathcal H^{\pm}$.
\end{lem}
\begin{proof}
1) Let $p_1,\dots,p_n$ be all the possible prime divisors of the torsion parts of $H_1(M,G_{\pm})$. If a prime $q$ is not in this set, then the short exact sequence 
$$0\to \mathcal H^{\pm}/\mathcal V^{\pm}\to H_1(M)\to H_1(M,G_{\pm})\to0$$
remains short exact after tensoring with $\mathbb Z/q\mathbb Z$. Using the Universal Coefficient Theorem, we conclude that 
\begin{equation}\label{eq:H/V}(\mathcal H^{\pm}/\mathcal V^{\pm})\otimes\mathbb Z/q\mathbb Z=\mathcal H^{\pm}/(\mathcal V^{\pm}+q\mathcal H^{\pm})
\end{equation} 
is the kernel of the map $H_1(M;\mathbb Z/q\mathbb Z)\to H_1(M,G_{\pm};\mathbb Z/q\mathbb Z)$, which is the image of the map
$\iota_*\co H_1(G_{\pm};\mathbb Z/q\mathbb Z)\to H_1(M;\mathbb Z/q\mathbb Z)$. 

Consider the commutative diagram
$$
\xymatrix{
H_1(G_{\pm})\ar[rr]\ar[d]&&H_1(G_{\pm};\mathbb Z/q\mathbb Z)\ar[d]^{\iota_*}\\
\mathcal H^{\pm}/\mathcal V^{\pm}\ar[r]&(\mathcal H^{\pm}/\mathcal V^{\pm})\otimes\mathbb Z/q\mathbb Z\;\ar@{^{(}->}[r]&H_1(M;\mathbb Z/q\mathbb Z).
}
$$
If $c\in H_1(G_{\pm})$ represents an element in $\ker \iota_*$, then the image of $c$ in the group (\ref{eq:H/V}) is zero, hence our conclusion follows.

2) The  image of the map
$$H_2(M,\partial M)\to H_1(\partial M)$$
is $\mathcal K$.
Using the isomorphism $H_1(\partial M,G_{\pm})\cong H_1(G_{\mp})$, we can identify the image of the map
$$H_2(M,\partial M)\to H_1(\partial M,G_{\pm})$$
with $\mathrm{Pr}_{\mp}(\mathcal K)$.
 By the exact sequence
$$H_2(M,\partial M)\to H_1(\partial M,G_{\pm})\to H_1(M,G_{\pm}),$$
the cokernel of the map $$\mathrm{Pr}_{\mp}\co\mathcal K\to \mathcal H^{\mp}$$ is a subgroup of $H_1(M,G_{\pm})$. In particular, any prime divisor of the torsion part of  $\mathrm{coker}\:\mathrm{Pr}_{\pm}$ must be one of $p_1,\dots,p_n$. Consider the short exact sequence
$$0\to \mathrm{Pr}_{\mp}(\mathcal K)\to \mathcal H^{\mp}\to \mathrm{coker}\:\mathrm{Pr}_{\mp}\to 0.$$
If an element $c\in \mathrm{Pr}_{\mp}(\mathcal K)$  is primitive in $\mathrm{Pr}_{\mp}(\mathcal K)$ but not primitive in $\mathcal H^{\mp}$, and $c$, as an element in $\mathcal H^{\mp}$, is not divisible by any $p_i$, then $c$ must be divisible (in $\mathcal H^{\mp}$) by some $q\notin\{p_1,\dots,p_n\}$. Thus $c$ represents a nonzero element in the kernel of the map
$$i_q\co(\mathrm{Pr}_{\mp}(\mathcal K))\otimes\mathbb Z/q\mathbb Z\to \mathcal H^{\mp}\otimes \mathbb Z/q\mathbb Z.$$
Since $q$ is not a divisor of the torsion part of $\mathrm{coker}\:\mathrm{Pr}_{\mp}$, $$\mathrm{Tor}(\mathrm{coker}\:\mathrm{Pr}_{\mp},\mathbb Z/q\mathbb Z)=0,$$ so $i_q$ is injective, a contradiction.
\end{proof}

Let $\upsilon=\dim\mathcal V^+_{\mathbb Q}$. 

\begin{lem}\label{lem:CritPrim}
Suppose that $\upsilon=0$. Let $b=(b^+,b^-)\in \mathcal K$. If none of $b^{\pm}$ is divisible by any $p_i$ in Lemma~\ref{lem:ExclPrime}, then $b=k(c^+,c^-)$ for some integer $k$ and primitive elements $c^{\pm}\in \mathcal H^{\pm}$.
\end{lem}
\begin{proof}
As $\upsilon=0$, $\mathcal V^{\pm}=0$.
Suppose $b=(kc^+,lc^-)$, where $k,l\in\mathbb Z$ and $c^{\pm}$ is primitive in $\mathcal H^{\pm}$. By the assumption $k,l$ are  not divisible by any $p_i$. Let $$d=\gcd(k,l), \quad k=k'd,l=l'd.$$ Then $b'=(k'c^+,l'c^-)$ is also contained in $\mathcal K$. If $|l'|>1$, $k'c^+=0\in H_1(M;\mathbb Z/l'\mathbb Z)$. Since $\gcd(k',l')=1$, $c^+=0\in H_1(M;\mathbb Z/l'\mathbb Z)$. It follows from Lemma~\ref{lem:ExclPrime} 1) that $c^+\in q\mathcal H^+$ for any prime $q|l'$, a contradiction to the assumption that $c^+$ is primitive in $\mathcal H^+$. This shows that $l'=\pm1$. Similarly, $k'=\pm1$. This finishes the proof.
\end{proof}

The next lemma was reminded to the author by Mirela \c{C}iperiani.

\begin{lem}\label{lem:Remainder}
Suppose $p_1,\dots,p_n$ are $n$ different primes, $r_i\in(\mathbb Z/p_i\mathbb Z)^m$. Let $P=p_1\cdots p_n$. Then there exists an element $x\in(\mathbb Z/P\mathbb Z)^m$ such that $$x\equiv r_i\pmod{p_i(\mathbb Z/P\mathbb Z)^m}.$$
\end{lem}
\begin{proof}
Apply the Chinese Remainder Theorem to each coordinate of $x$.
\end{proof}

\begin{lem}\label{lem:A_pDim}
Let $p$ be a prime number. Let $\mathcal{A}^{\pm}_p=\mathrm{Pr}^{-1}_{\pm}( p\mathcal H^{\pm})$, and let
$A^{\pm}_p$ be  the image of $\mathcal{A}^+_p$ under the map $\mathcal K\to\mathcal K/p\mathcal K$. Then $A^{\pm}_p$ has dimension at most $g$.
\end{lem}
\begin{proof}
As $\mathcal K$ is a direct summand of $\mathcal H$, $\mathcal K/p\mathcal K$ is a subspace of $\mathcal H/p\mathcal H$. 
If the dimension of $A^{\pm}_p$ is greater than $g$, we can suppose that $(b^+_1,b^-_1),\dots,(b^+_{g+1},b^-_{g+1})\in \mathcal{A}^+_p$ are $(g+1)$ elements, and they are linearly independent (over $\mathbb Z/p\mathbb Z$) modulo $p\mathcal H$. As each $b^+_i$ is contained in $p\mathcal H^+$, each $b^-_i$ represents an element $[b^-_i]\in\mathcal V^-_p$, and $[b^-_1],\dots,[b^-_{g+1}]$ are linearly independent. This contradicts Lemma~\ref{lem:LagSpace}, which implies that the dimension of $\mathcal V^-_p$ is at most $g$.
\end{proof}

\begin{proof}[Proof of Proposition~\ref{prop:HomoC}]
If $\upsilon>0$, then both $\mathcal V^+$ and $\mathcal V^-$ have positive rank. Let $c^{\pm}$ be a primitive element in $\mathcal H^{\pm}$ such that $mc^{\pm}\in \mathcal V^{\pm}$ for some positive integer $m$. Then $c^+$ is homologous to $c^-$ in $H_1(M;\mathbb Q)$.

Now consider the case $\upsilon=0$. By Lemma~\ref{lem:SurgTorsFree}, we may assume $H_1(M)$ is torsion-free.
Let $\{p_1,\dots,p_n\}$ be the set of primes as in Lemma~\ref{lem:ExclPrime}, and let $A^{\pm}_{p_i}$ be as in Lemma~\ref{lem:A_pDim}.
By Lemma~\ref{lem:A_pDim}, $(\mathcal K/p\mathcal K)\backslash(A^+_p\cup A^-_p)$ is always nonempty, so we can apply Lemma~\ref{lem:Remainder} to get $b=(b^+,b^-)\in\mathcal K$ such that its residue modulo $p_i\mathcal K$ is not contained in $A^{+}_{p_i}\cup A^{-}_{p_i}$ for each $i$. Thus none of $b^+,b^-$ is divisible by any $p_i$. By Lemma~\ref{lem:CritPrim}, $b=k(c^+,c^-)$ for some integer $k$ and primitive elements $c^{\pm}\in \mathcal H^{\pm}$.
\end{proof}


\section{Decreasing the reduced complexity}\label{Sect:DecCx}

In this section, we will prove that the reduced complexity for the pair $(Y,G)$ can always be decreased if $Y$ does not fiber over $S^1$ with $G$ being a fiber.

\begin{notn}
Throughout this section, $Y$ is a closed, oriented, connected, irreducible $3$--manifold, $G\subset Y$ is a connected taut surface with $g(G)>1$, and $M=Y\Bslash G$. Let $G_+,G_-$ be  the two components of $\partial M$, oriented so that $\partial M=G_+\sqcup(-G_-)$. Then $M$ has a natural sutured manifold structure with $\gamma=\emptyset$, $R_{\pm}(\gamma)=G_{\pm}$.
\end{notn}

\begin{defn}
Let $A$ be a product annulus in $M$.

$\bullet$ $A$ is of {\it type NN}, if both components of $\partial A$ are non-separating in $R(\gamma)$; 

$\bullet$ $A$ is of {\it type NS}, if $\partial_- A$ is non-separating and $\partial_+A$ is separating; 

$\bullet$ $A$ is of {\it type SN}, if $\partial_- A$ is separating and $\partial_+A$ is non-separating; 

$\bullet$ $A$ is of {\it type SS}, if both components of $\partial A$ are separating.
\end{defn}



\begin{figure}
\begin{picture}(300,95)
\put(30,0){\scalebox{0.5}{\includegraphics*
{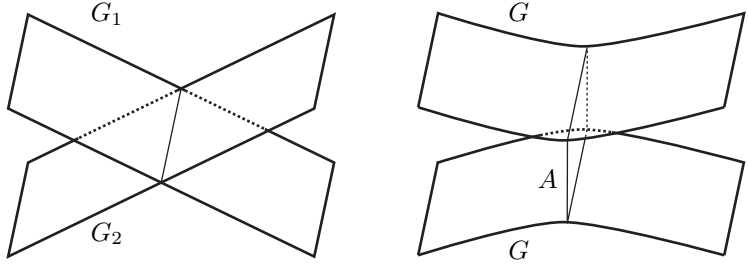}}}

\put(220,90){$G$}

\put(62,90){$G_1$}

\put(220,0){$G$}

\put(231,27){$A$}

\put(62,7){$G_2$}

\end{picture}
\caption{The cut-and-paste construction creates product annuli or disks}\label{fig:CutPaste}
\end{figure}

\begin{construction}\label{constr:CutPaste}
Recall that the cut-and-paste construction starts with two properly embedded oriented surfaces $G_1, G_2\subset Y$ with $G_1\pitchfork G_2$, and results in a properly embedded oriented surface $G$. Figure~\ref{fig:CutPaste} shows a local picture of the cut-and-paste, where all surfaces are oriented by upward normal vectors. Suppose $C$ is a component of $G_1\cap G_2$, then there exists an annulus $A=C\times I$, such that $\partial A=G\cap A$, and different ends of $A$ approach $G$ from different sides. Let $M=Y\Bslash G$, then $A$ becomes a product annulus in $M$.
\end{construction}

\begin{construction}\label{constr:Reglue}
Our basic construction is as follows. Suppose that $S\subset M$ is a taut surface such that there exists a homeomorphism  $f\co G_+\to G_-$ which sends $\partial_+S$ to $\partial_-S$. Let $(Y',\overline S)$ be the pair obtained from $(M,S)$ by gluing $G_+$ to $G_-$ via $f$. Let $G'\subset Y'$ be the surface obtained from $\overline S$ and $G$ by cut-and-pastes. Let $(M',\gamma')$ be the sutured manifold corresponding to the pair $(Y',G')$. As illustrated in Figure~\ref{fig:BasicS}, there is a commutative diagram of sutured manifold decompositions\footnote{This commutative diagram of topological manifold decompositions is obvious from Figure~\ref{fig:BasicS}. To prove it for sutured manifold decompositions, we need to figure out the change of the sutures, which is left to the reader.}:
\begin{equation}\label{eq:CommDiag}
\xymatrix{
Y'\ar@{~>}[r]^-{G'}\ar@{~>}[d]^{G}&(M',\gamma')\ar@{~>}[d]^{A}\\
(M,\gamma)\ar@{~>}[r]^-{S}&(M_1,\gamma_1),
}
\end{equation}
where $A$ is a collection of disjoint product annuli in $M'$ whose components correspond to the components of $\partial_+S$, as in Construction~\ref{constr:CutPaste}. As $M_1$ is taut, Lemma~\ref{lem:ProdIff} implies that $M'$ is also taut, hence $G'$ is taut in $Y'$.
\end{construction}

\begin{figure}
\begin{picture}(300,165)
\put(30,0){\scalebox{0.7}{\includegraphics*
{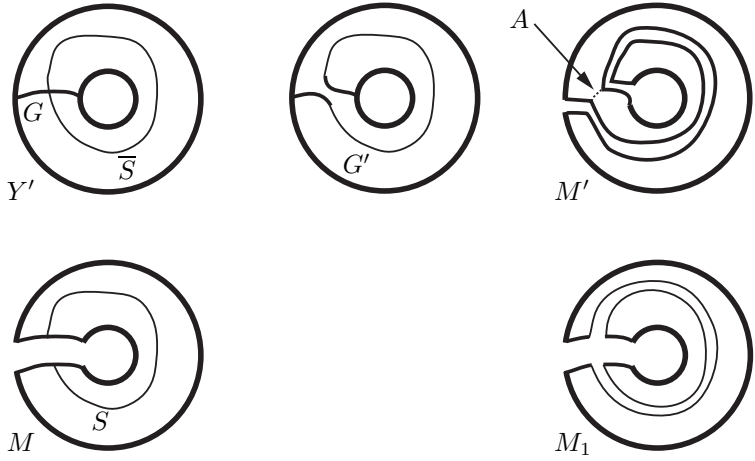}}}

\put(218,160){$A$}

\put(235,95){$M'$}

\put(28,95){$Y'$}

\put(28,0){$M$}

\put(235,0){$M_1$}

\put(34,125){$G$}

\put(155,105){$G'$}

\put(60,8){$S$}

\put(70,103){$\overline S$}

\end{picture}
\caption{Illustration of the commutative diagram (\ref{eq:CommDiag})}\label{fig:BasicS}
\end{figure}

\begin{rem}\label{rem:OneCurve}
The simplest case where there exists a homeomorphism  $f\co G_+\to G_-$ which sends $\partial_+S$ to $\partial_-S$ is that there exist a positive integer $k$ and non-separating simple closed curves $C_{\pm}\subset G_{\pm}$ such that $\partial_{\pm}S$ consists of $k$ parallel copies of $C_{\pm}$. Moreover, we can assume that $k$ is the smallest positive integer such that $k[C_+]=k[C_-]\in H_1(M)$, and $S$ has no closed components. Throughout this paper the above conditions will always be the case whenever we apply Construction~\ref{constr:Reglue}.
\end{rem}

\begin{lem}\label{lem:G'Connected}
Suppose that $S$ is a surface as in Remark~\ref{rem:OneCurve}. Let $G^{(m)}$ be the surface obtained from $\overline S$ and $m$ copies of $G$ by cut-and-pastes, then $G^{(m)}$ is connected.
\end{lem}
\begin{proof}
We first prove that $G'$ is connected.
In the new manifold $Y'$, $\overline S\cap G$ consists of $k$ curves, each of which is parallel to $C\subset G$. Then $G\Bslash(\overline S\cap G)$ has $k$ components: $A_1,\dots,A_{k-1},G\Bslash C$, where each $A_i$ is an annulus. Assume that $G'$ is not connected, let $G_1$ be a component of $G'$ which does not contain $G\Bslash C$. Let $A=G_1\cap(A_1\cup\cdots\cup A_{k-1})\ne\emptyset$, $S_1=G_1\backslash\mathrm{int}(A)$. Then $S_1$ is a union of components of $S$, and $|\partial_-S_1|=|\partial_+S_1|=|A|\le k-1$, a contradiction to the assumption on $k$ in Remark~\ref{rem:OneCurve}. So $G'$ is connected.

If $G^{(m)}$ is connected, then the result in the last paragraph shows that $G^{(m+1)}$ is also connected. This finishes the proof.
\end{proof}

\begin{lem}\label{lem:ExSpinc}
In Construction~\ref{constr:Reglue}, if $S$ satisfies the condition in Remark~\ref{rem:OneCurve}, then $(Y',G')\prec(Y,G)$.
\end{lem}
\begin{proof}
In light of Lemma~\ref{lem:G'Connected}, we only need to check (\ref{eq:SpincSub}), which can be achieved by applying Lemma~\ref{lem:h1h2} 1) to $[\overline S],[G]\subset H_2(Y')$.
\end{proof}

\begin{figure}
\begin{picture}(300,135)
\put(30,0){\scalebox{0.6}{\includegraphics*
{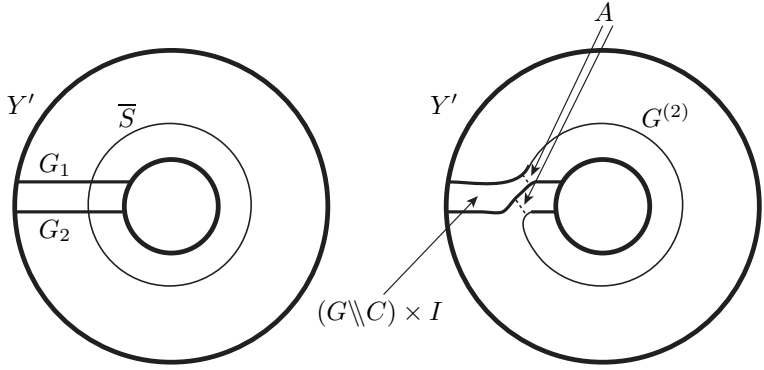}}}

\put(28,95){$Y'$}

\put(70,90){$\overline S$}

\put(40,73){$G_1$}

\put(40,48){$G_2$}

\put(268,90){$G^{(2)}$}

\put(188,95){$Y'$}

\put(145,17){$(G\Bslash C)\times I$}

\put(250,130){$A$}

\end{picture}
\caption{Decomposing $Y'$ along $G^{(2)}$ creates a sutured manifold with large product part}\label{fig:ParCopies}
\end{figure}

\begin{construction}\label{constr:CPtorus}
Suppose that  $S\subset M$ is a product annulus of type NN. We can apply Construction~\ref{constr:Reglue}. $\overline S$ is a torus with $\overline S\cap G=C$. Performing cut-and-pastes to $\overline S$ and two parallel copies of $G\subset Y'$, we get a new taut surface $G^{(2)}$. See Figure~\ref{fig:ParCopies} for an illustration of $G^{(2)}$. Abstractly, the surface $G^{(2)}$ is a two-fold cover of $G$ dual to $[C]\in H_1(G;\mathbb Z/2\mathbb Z)$. Let $M^{(2)}$ be the sutured manifold $Y'\Bslash G^{(2)}$, then we can find a union of two product annuli, denoted $A^{(2)}$, in $M^{(2)}$ as in Construction~\ref{constr:CutPaste}. We have the decomposition 
\begin{equation}\label{eq:CPDecomp}
M^{(2)}\stackrel{A^{(2)}}{\rightsquigarrow} \big((G\Bslash C)\times I\big)\sqcup \big(M\Bslash S\big).
\end{equation}
In other words, $M^{(2)}$ is obtained from $(G\Bslash C)\times I$ and $M\Bslash S$ by gluing the two copies of $C\times I$ to the two copies of $S$.
 It follows that $\mathscr R(Y'\Bslash G^{(2)})=\mathscr R(Y\Bslash G)$.
\end{construction}

\begin{construction}\label{constr:Subsurface}
Suppose that $P\subset G$ is an embedded essential subsurface, and that $\partial P$ consists of two unions of simple closed curves $C_0,C_1$, oriented so that $\partial P=-C_0\sqcup C_1$.
We can lift $P$ to a properly embedded surface $\widetilde P\subset G\times I$, such that $\partial\widetilde P=-(C_0\times0)\sqcup(C_1\times1)$. For example, let $$\widetilde P=(C_0\times[0,\frac12])\cup(P\times\{\frac12\})\cup(C_1\times[\frac12,1]).$$ Let $\widetilde P=\widetilde{\iota}(P)\subset G\times I$. Let $\widetilde G$ be the surface constructed by gluing two copies of $P$ (denoted $P_0,P_1$,) and $G\Bslash P$ together such that $G\Bslash P$ and $P_i$ are glued along $C_i$ for $i=0,1$. The boundary components of $\widetilde G$ correspond to the boundary components of $P$. The manifold obtained by decomposing $G\times I$ along $\widetilde P$ is the product sutured manifold $\widetilde G\times I$. This simple fact is illustrated in Figure~\ref{fig:Subsurface}. If $C_0,C_1\ne\emptyset$ and at least one of $C_0,C_1$ is non-separating, then $\widetilde G$ is connected.
\end{construction}

\begin{figure}
\begin{picture}(300,122)
\put(30,0){\scalebox{0.5}{\includegraphics*
{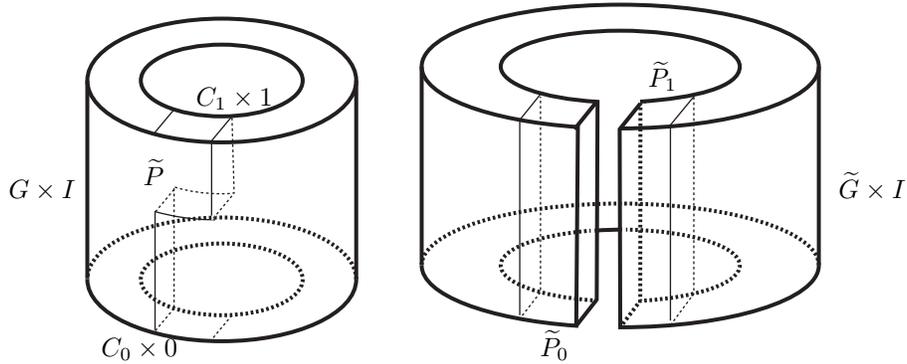}}}

\put(1,55){$G\times I$}

\put(315,55){$\widetilde G\times I$}

\put(52,60){$\widetilde P$}

\put(72,90){$C_1\times1$}

\put(36,-5){$C_0\times0$}

\put(202,-5){$\widetilde P_0$}

\put(243,97){$\widetilde P_1$}

\end{picture}
\caption{Illustration of the decomposition along $\widetilde P$. $\widetilde G\times \{0,1\}$ contains two copies of $\widetilde P$: $\widetilde P_0\subset \widetilde G\times0$ and $\widetilde P_1\subset \widetilde G\times1$.}\label{fig:Subsurface}
\end{figure}

\begin{lem}\label{lem:OneNN}
Suppose $\mathscr{R}(M,\gamma)\ne\emptyset$.
If $(M,\gamma)$ contains a product annulus of type NN, then there exists a pair $(Y',G')\prec(Y,G)$ satisfying that, for the sutured manifold $(M',\gamma')$ corresponding to $(Y',G')$, $M'\Bslash \mathscr{R}(M',\gamma')$ has exactly one component. 
\end{lem}
\begin{proof}
Suppose $(N_1,\delta_1)$ is a component of $M\Bslash \mathscr{R}(M,\gamma)$ that contains a product annulus $S$ of type NN in $M$. As $g(G)>1$, performing Construction~\ref{constr:CPtorus} to $S$ if necessary, we may assume $g(R_-(\delta_1))$ is positive\footnote{If we perform Construction~\ref{constr:CPtorus} to $S$, by  (\ref{eq:CPDecomp}) the new sutured manifold $M^{(2)}$ contains a product sutured manifold  $(N_1\Bslash S)\cup_{A^{(2)}}((G\Bslash C)\times I)$, whose $R_-$ has positive genus.}. Then we can find a product annulus $A_1\subset N_1$ which is of type NN in $N_1$. Suppose $(N_2,\delta_2)$ is another component of $M\Bslash \mathscr{R}(M,\gamma)$, let $A_2$ be a component of $\delta_2$. Let $A=A_1\cup A_2$, $N=M\Bslash A$. We can find a non-separating curve $C_{\pm}\subset G_{\pm}$ such that $\mp\partial_{\pm}A$ and $\pm C_{\pm}$ cobound a subsurface $P_{\pm}$ in $G_{\pm}$. In fact, $P_{\pm}$ can be chosen to be a pair of pants which is a neighborhood of $\partial_{\pm} A_1,\partial_{\pm} A_2$ and an arc connecting them. The condition that $C_{\pm}$ is non-separating is guaranteed by the fact that $\partial_{\pm}A_1$ is non-separating in $R_{\pm}(\delta_1)$ while $\partial_{\pm}A_2$ is disjoint from $R_{\pm}(\delta_1)$.

We apply Construction~\ref{constr:Subsurface} to $P_{\pm}$ to get taut decompositions $$G\times[-1,0]\stackrel{\widetilde{P}_-}{\rightsquigarrow}\widetilde G_-\times[-1,0],\quad G\times[1,2]\stackrel{\widetilde{P}_+}{\rightsquigarrow}\widetilde G_+\times[1,2].$$ 
Then $\widetilde G_{\pm}$ is connected. We can glue the pairs $(G\times[-1,0],\widetilde{P}_-),(G\times[1,2],\widetilde{P}_+)$ to $(M,A)$ such that $G\times0$ is glued to $G_-$ and $G\times1$ is glued to $G_+$, then we get a pair $(M^*,S)$, where $M^*$ is homeomorphic to $M$. See Figure~\ref{fig:TwoBOne} for a schematic picture. 

Let $f\co G\times2\to G\times(-1)$ be a homeomorphism which sends $C_+\times2$ to $C_-\times(-1)$. Let $(Y_1,\overline S)$ be the pair obtained from $(M^*,S)$ by gluing via $f$.
Let $G_1\subset Y_1$ be the surface obtained from $\overline S,G\times0,G\times1,G\times2$ by cut-and-pastes, $M_1=Y_1\Bslash G_1$. Then we have a commutative diagram of decompositions
$$\xymatrixcolsep{5pc}
\xymatrix{
Y_1\ar@{~>}[r]^-{G\times\{0,1,2\}}\ar@{~>}[d]^{G_1}&(G\times[-1,0])\sqcup M\sqcup (G\times[1,2])\ar@{~>}[d]^{\widetilde{P}_-\sqcup A\sqcup\widetilde{P}_+}\\
M_1\ar@{~>}[r]^-{(\overline S\cap (G\times\{0,1,2\}))\times I}&(\widetilde G_-\times[-1,0])\sqcup N\sqcup (\widetilde G_+\times[1,2]).
}
$$
So $M_1$ is obtained from $N$ by gluing $(\widetilde G_-\cup_{C_-=C_+} \widetilde G_+)\times I$ to $N$ along the two copies of $A$ in $\partial N$. Here $C_{\pm}$ is viewed as a component of $\partial\widetilde  G_{\pm}$. Thus $M_1$ is taut, and hence $G_1$ is taut. 

Clearly $\mathscr{R}(M_1)=\mathscr{R}(N)=\mathscr{R}(M)$. Any component of $M_1\Bslash \mathscr{R}(M_1)$ contains at least one component of $M\Bslash \mathscr{R}(M)$. Moreover, as $\widetilde G_{\pm}$ is connected, the two copies of $A_1$ and $A_2$ are in the same component of $M_1\Bslash \mathscr{R}(M_1)$, so  $$|M_1\Bslash \mathscr{R}(M_1)|<|M\Bslash \mathscr{R}(M)|.$$
Repeat the above procedure until we get a pair $(Y_n,G_n)$ such that the corresponding sutured manifold $(M_n,\gamma_n)$ has $|M_n\Bslash\mathscr{R}(M_n)|=1$. This $(Y_n,G_n)$ is the pair we want. 
\end{proof}

\begin{figure}
\begin{picture}(240,150)
\put(100,0){\scalebox{0.7}{\includegraphics*
{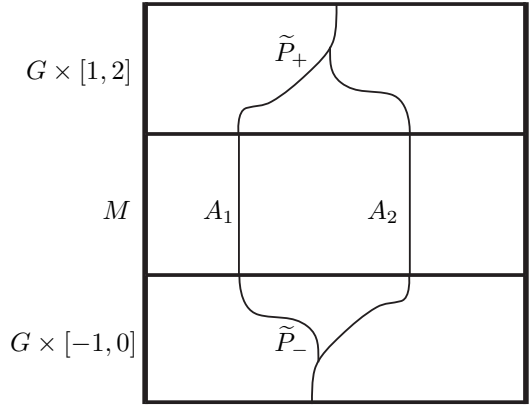}}}

\put(123,70){$A_1$}

\put(185,70){$A_2$}

\put(150,20){$\widetilde P_-$}

\put(56,123){$G\times[1,2]$}

\put(150,130){$\widetilde P_+$}

\put(85,70){$M$}

\put(50,20){$G\times[-1,0]$}

\end{picture}
\caption{Schematic picture of $A$ and $\widetilde P_{\pm}$}\label{fig:TwoBOne}
\end{figure}

\begin{lem}\label{lem:GluePDisk}
If $\mathscr{R}(M,\gamma)\ne\emptyset$, and $(M,\gamma)$ satisfies that $M\Bslash \mathscr{R}(M,\gamma)$ is connected and contains a product annulus of type NN, then there exists a pair $(Y',G')\prec(Y,G)$ such that $C(Y',G')<C(Y,G)$.
\end{lem}
\begin{proof}
Our argument here is similar to the argument used by Kronheimer and Mrowka \cite[Proposition~6.9]{KMSuture}.
Let $(N_1,\delta_1)=\mathscr{R}(M,\gamma)$, and let $(N_2,\delta_2)$ be the product sutured manifold $M\Bslash N_1$. Since $H_1(\partial N_1)\ne0$, we can always find a curve $C\subset \partial N_1$ such that $[C]\ne0\in H_1(\partial N_1)$ and $C$ is null-homologous in $N_1$. We can then isotope $C$ so that $C\cap R(\delta_1)$ is efficient in $R(\delta_1)$. By Theorem~\ref{thm:ExistTaut}, there exists a connected taut decomposing surface $T_1\subset N_1$, such that $T_1$ is connected and $0\ne[\partial T_1]\in H_1(\partial N_1)$. By Juh\'asz \cite[Lemma~4.5]{Ju2}, we can isotope $T_1$ to be a decomposing surface $T_1'$, such that every component of $\partial T_1'$ intersects both $R_+(\delta_1)$ and $R_-(\delta_1)$, and the decompositions along $T_1$ and $T_1'$ give us the same sutured manifold. Since $\delta_1$ separates $\partial N_1$ into two parts, the algebraic intersection of $\partial T_1'$ with the sutures $s(\delta_1)$ is zero. We can assume $\partial T_1'\cap \delta_1$ consists of $2n$ vertical arcs, then $\partial_+ T_1'$ consists of $n$ arcs $a_1^+,\dots,a_n^+$, and $\partial_- T_1'$ consists of $n$ arcs $a_1^-,\dots,a_n^-$. 

Since $g(G)>1$, applying Construction~\ref{constr:CPtorus} to a product annulus of type NN, we can increase the genus of $R_-(\delta_2)$ as large as possible. Construction~\ref{constr:CPtorus} does not change the number $n$, because $n$ is determined by $T_1'\subset N_1$ while Construction~\ref{constr:CPtorus} does not change $(N_1,\delta_1)=\mathscr{R}(M,\gamma)$. So we may assume $g(R_-(\delta_1))>\frac n2$.
Now since $N_2$ is connected, we can find disjoint arcs $b_1^-,\dots,b_n^-\subset R_-(\delta_2)$, such that $b_i^-$ connects the endpoints of $a_i^-$ and $R_-(\delta_2)-(b_1^-\cup\cdots\cup b_n^-)$ is connected. 

Let $S_1$ be the union of $T_1'$ and the product disks $b_1^-\times I,\dots,b_n^-\times I\subset N_2$, then $\partial_{\pm}S_1$ consists of $n$ simple closed curves whose homology classes are linearly independent in $H_1(R_{\pm}(\gamma))$. The decomposition $(M,\gamma)\stackrel{S_1}{\rightsquigarrow}(M_1,\gamma_1)$ is taut, and $$\overline{C}(M_1,\gamma_1)=\overline{C}(N_1\Bslash T_1)<C(N_1,\delta_1)=\overline{C}(M,\gamma),$$ where the inequality follows by applying Lemma~\ref{lem:Complexity} to the decomposition of $(N_1,\delta_1)$ along $T_1$.

At this stage we can construct $(Y',G')$ by gluing $G_+$ to $G_-$ via a homeomorphism which sends $\partial_+S_1$ to $\partial_- S_1$. However, in order to comply with Remark~\ref{rem:OneCurve}, we need to do more. We can choose a subsurface $P_{\pm}\subset G_{\pm}$ such that $\partial P_{\pm}=(\pm C_{\pm})\sqcup(\mp \partial_{\pm}S_1)$, where $C_{\pm}$ is a non-separating simple closed curve in $G_{\pm}$. We can glue $P_-,P_+$ to $S_1$, then do cut-and-pastes with $G_-,G_+$, thus get a taut surface $S\subset M$ with $\partial_{\pm}S=C_{\pm}$. Now we can apply Construction~\ref{constr:Reglue} to $S\subset M$ to get a pair $(Y',G')$. As in the proof of Lemma~\ref{lem:OneNN}, we can show that $\mathscr{R}(M\Bslash S)=\mathscr{R}(M_1)$. So $$C(Y',G')=\overline{C}(M\Bslash S)=\overline{C}(M_1,\gamma_1)<\overline{C}(M,\gamma)=C(Y,G).$$
\end{proof}

\begin{lem}\label{lem:NS+SN}
Suppose that $(M,\gamma)$ does not contain product annuli of type NN, while it contains both product annuli of type NS and product annuli of type SN. Then there exists a taut connected surface $S\subset M$ such that each of $\partial_-S$ and $\partial_+S$ consists of a single non-separating simple closed curve.
\end{lem}
\begin{proof}
Let $(N_1,\delta_1)$ be a component of the product sutured manifold $M\Bslash \mathscr{R}(M,\gamma)$. We claim that if $N_1$ contains a product annulus of type NS (or type SN) in $M$, then at least one component of $\delta_1$ is a product annulus of type NS (or type SN) in $M$.  
If fact,
suppose $A_1\subset N_1$ is a product annulus of type NS in $M$, then $[\partial_-A_1]\ne0$ in $H_1(G_-)$. Since $M$ does not contain product annuli of type NN, $R_-(\delta_1)$ must be planar, otherwise the product sutured manifold $(N_1,\delta_1)$ must contain a type NN product annulus, which must also be of type NN in $M$. So $[\partial_-A_1]$ is equal to the sum of some components of $\partial R_-(\delta_1)$. It follows that there is a component $C_{1-}$ of $\partial R_-(\delta_1)$ with $[C_{1-}]\ne0$ in $H_1(G_-)$. So $C_{1-}$ is non-separating in $G_-$. Since $M$ does not contain annuli of type NN, the annulus $C_{1-}\times I\subset\delta_1$ must be of type NS. The same argument works for type SN annuli.

By the last paragraph, we conclude that $\partial_v\mathscr{R}(M)$ contains a product annulus $A_1$ of type NS and a product annulus $A_2$ of type SN, thus the two annuli $A_1,A_2$ are disjoint. We observe that $[\partial_{\pm}A_1]+[\partial_{\pm}A_2]$ is always a primitive element in $H_1(G_{\pm})$. So we can find a non-separating curve $C_{\pm}\in G_{\pm}$ which cobounds a subsurface $\pm P_{\pm}$ with $-\partial_{\pm}A_1,-\partial_{\pm}A_2$. Gluing $P_-,A_1,A_2$ and $P_+$ together we get a connected surface $S$ as we did in the proof of Lemma~\ref{lem:OneNN}. Consider the decomposition $M\stackrel{S}{\rightsquigarrow}M'$. We can decompose $M'$ along $(G_-\cup G_+)\backslash S$ to get $$((G_-\times I)\Bslash\widetilde{P}_-)\sqcup(M\Bslash(A_1\cup A_2))\sqcup((G_+\times I)\Bslash\widetilde{P}_+)$$ which is taut. By Lemma~\ref{lem:RevTaut}, the decomposition $M\stackrel{S}{\rightsquigarrow}M'$ is taut.
\end{proof}

\begin{lem}\label{lem:NoNN+NS}
Suppose $(M,\gamma)$ does not contain product annuli of type NN or type NS, then there exists a taut decomposition $(M,\gamma)\stackrel{S}{\rightsquigarrow}(M',\gamma')$ and a positive integer $k$ such that $\partial_{\pm}S$ consists of $k$ parallel non-separating curves, $S$ satisfies Remark~\ref{rem:OneCurve}, and $\overline{C}(M',\gamma')<\overline{C}(M,\gamma)$.
\end{lem}
\begin{proof}
By Proposition~\ref{prop:HomoC} and Theorem~\ref{thm:ExistTaut}, we can find a taut surface $S'$ and a positive integer $k$ such that $\partial_{\pm}S'$ consists of $k$ parallel non-separating curves, and $S'$ satisfies Remark~\ref{rem:OneCurve}. 

Let $\rho=\partial_v\mathscr{R}(M)$, $M_2=M\Bslash\mathscr{R}(M)$.
We can modify $S'$ to get a new decomposing surface $S$ with $\partial_{\pm} S$ being isotopic to $\partial_{\pm}S'$ in $G_{\pm}$, such that $S$ is still taut and the intersection of $S$ with each component of $\rho$ consisting of either parallel oriented essential arcs or parallel oriented essential closed curves. 
Let $S_1=S\cap \mathscr{R}(M), S_2=S\Bslash S_1$. 

We claim that $[S_1\cap\partial\mathscr{R}(M,\gamma)]\ne0$ in $H_1(\partial\mathscr{R}(M,\gamma))$. This claim is clearly true if $S$ intersects a component of $\rho$ in parallel oriented essential arcs. Now we assume that $S\cap\rho$ consists of essential closed curves, then $\partial_-S$ is disjoint with $\rho$. As $M$ does not contain product annuli of type NN or type NS, every component of $\partial_-\rho$ is separating in $G_-$, hence any component of $\partial_-S$ is non-separating in the component of $G_-\Bslash(\partial_-\rho)$ containing it. Using the fact that $M$ does not contain product annuli of type NN again, we see that $\partial_- S\subset R_-(\mathscr{R}(M,\gamma))$. So $[\partial_-S]\ne0$ in $H_1(R_-(\mathscr{R}(M,\gamma)))$, hence our claim holds.

We have a commutative diagram of sutured manifold decompositions:
$$
\xymatrix{
(M,\gamma)\ar@{~>}[r]^-{\rho}\ar@{~>}[d]^{S}&\mathscr{R}(M)\sqcup(M_2,\rho)\ar@{~>}[d]^{S_1\sqcup S_2}\\
(M',\gamma')\ar@{~>}[r]^-{\rho\Bslash (S\cap\rho)}&(M_1',\gamma_1')\sqcup(M_2',\gamma_2').
}$$
Here the decomposing surface $\rho\Bslash (S\cap\rho)$ consists of product disks and product annuli. By Lemma~\ref{lem:ProdIff}, $(M_1',\gamma_1')\sqcup(M_2',\gamma_2')$ is taut.
By \cite[Lemma~2.4]{GabaiFibred}, $(M_2',\gamma_2')$ is a product sutured manifold. So $\overline{C}(M',\gamma')=\overline{C}(M_1',\gamma_1')$. By Lemma~\ref{lem:Complexity} and the claim in the last paragraph, $$\overline{C}(M_1',\gamma_1')<C(\mathscr{R}(M))=\overline{C}(M,\gamma).$$
So $\overline{C}(M',\gamma')<\overline{C}(M,\gamma)$.
\end{proof}

Now we are ready to prove Theorem~\ref{thm:PreciseGabai}.

\begin{proof}
[Proof of Theorem~\ref{thm:PreciseGabai}]
Let $G$ be a connected taut surface in $Y$. If $\mathscr{R}(M,\gamma)=\emptyset$, then $(M,\gamma)$ is a product sutured manifold and we are done. 

If $\mathscr{R}(M,\gamma)\ne\emptyset$, we claim that we can find a $(Y',G')\prec(Y,G)$ such that $C(Y',G')<C(Y,G)$.
If $M$ contains a product annulus of type NN, we can apply Lemmas~\ref{lem:OneNN} and \ref{lem:GluePDisk}. From now on we assume $M$ does not contain product annuli of type NN. If $M$ contains both product annuli of type NS and product annuli of type SN, then we can  get a surface $S$ as in Lemma~\ref{lem:NS+SN}. Apply Construction~\ref{constr:Reglue} to $S\subset M$, we get a $(Y',G')\prec(Y,G)$ such that $(M',\gamma')$ contains a product annulus of type NN, as explained in Construction~\ref{constr:CutPaste}. This reduces to the previous case. If $M$ does not contain product annuli of type NS, we can apply Lemma~\ref{lem:NoNN+NS} then use Construction~\ref{constr:Reglue}. The same argument works if $M$ does not contain product annuli of type SN. This finishes the proof of the claim.

Now we work with the pair $(Y',G')$ and repeat the above procedure until we get a pair such that the corresponding sutured manifold is a product. This proves our theorem.
\end{proof}


\section{Heegaard Floer homology, Thurston norm, and fibrations}\label{Sect:Norm}

In this section, we will review the results about Heegaard Floer homology, Thurston norm, and fibrations. Using Gabai's internal hierarchy, we will take a new look at these results. As a consequence, we improve these results by taking account of the homological action.

\subsection{Review of the results}

Basic classes (Definition~\ref{defn:Basic}) are closely related to the Thurston norm. In fact, \cite[Theorem~1.1]{OSzGenus} implies that the support of the basic classes (for a twisted version of Heegaard Floer homology) determines the Thurston norm. The following is a statement of this theorem for untwisted Heegaard Floer homology (see \cite[Theorem~2.3]{NiNormCos}).

\begin{thm}\label{thm:ThNorm}
Suppose that $Y$ is a closed oriented $3$--manifold, $h\in H_2(Y)$, then $\mathcal B_Y(h)\ne\emptyset$.
\end{thm}

Suppose that $K\subset Y$ is a rationally null-homologous knot. Let $X=Y\Bslash K$. 

\begin{defn}
A cohomology class $\alpha\in H^2(Y,K)\cong H^2(X,\partial X)$ is a {\it(knot Floer) basic class}, if $\alpha=c_1(\xi)$ for some $\xi\in\relspin(Y,K)$ with $\widehat{HFK}(Y,K,\xi)\ne0$. Let $\underline{\mathcal B}_{(Y,K)}$ be the set of all basic classes. Let $j^*\co H^2(Y,K)\to H^2(Y)$ be the pull-back map. We define $\mathcal B_{(Y,K)}=j^*(\underline{\mathcal B}_{(Y,K)})$.
\end{defn}

\begin{defn}
Suppose that $K$ is an oriented  rationally null-homologous knot in a closed $3$--manifold
$Y$. A properly embedded oriented surface $F\subset Y\Bslash K$ is a {\it rational Seifert-like
surface} for $K$, if $\partial F$ consists of a nonzero number of parallel essential curves on $\partial\nu(K)$, such that the orientation of $\partial F$ is coherent with the orientation of $K$. When $F$ contains no closed components, we say
that $F$ is a {\it rational Seifert surface} for $K$. 
\end{defn}

\begin{prop}\label{prop:ExtremeSpinc}
Suppose that $K$ is an oriented rationally null-homologous knot in $Y$ such that $\partial\nu(K)$ is incompressible in $X$. Let $F$ be a rational Seifert-like surface for $K$. Then
\begin{eqnarray*}
\min_{\alpha\in\underline{\mathcal B}_{(Y,K)}}\{\langle\alpha,[F]\rangle\}&=&-\chi_-([F]),\\
\max_{\alpha\in\underline{\mathcal B}_{(Y,K)}}\{\langle\alpha,[F]\rangle\}&=&\chi_-([F])+2|[\partial F]\cdot[\mu]|,
\end{eqnarray*}
where $\mu\in \partial\nu(K)$ is the meridian of $K$.
\end{prop}
\begin{proof}
This is a standard result, although not explicit in the literature. The reader is referred to Ni \cite[Theorem~2.4]{NiNormCos} for the case when $K$ is null-homologous, and to Ni \cite[Theorem~3.1]{NiRatSphere}, Hedden \cite{HedCable} and Ozsv\'ath--Szab\'o\cite{OSzLinkNorm} for the procedure of passing from null-homologous knots to  rationally null-homologous knots. The apparent asymmetry between the min and the max actually reflects the symmetry in knot Floer homology, see Ozsv\'ath--Szab\'o\cite[Proposition~8.2]{OSzLink}.
\end{proof}

\begin{defn}
Suppose that $\varphi\in H_2(Y,K)$ is a homology class. Let $$\underline{\mathcal B}_{(Y,K)}(\varphi)=\left\{\alpha\in\underline{\mathcal B}_{(Y,K)}\big|\:\langle \alpha,\varphi\rangle=-\chi_-(\varphi)\right\}$$
be the set of {\it bottommost (relative knot Floer) basic classes} on $X$ with respect to $\varphi$.
When $h\in H_2(Y)$, let $j_*\co H_2(Y)\to H_2(Y,K)=H_2(X,\partial X)$ be the natural map, then let $${\mathcal B}_{(Y,K)}(h)=j^*\underline{\mathcal B}_{(Y,K)}(j_*(h))$$
be the set of {\it bottommost (knot Floer) basic classes} on $Y$ with respect to $h$.
\end{defn}

Thus Proposition~\ref{prop:ExtremeSpinc} says that $\underline{\mathcal B}_{(Y,K)}(\varphi)\ne\emptyset$ for any $\varphi\in H_2(Y,K)$ representing a rational Seifert-like surface. There is a similar statement for the homology classes of closed surfaces.

\begin{prop}\label{prop:ThNormK}
Suppose that $K$ is an oriented rationally null-homologous knot in $Y$. Then for any $h\in H_2(Y)$, we have
$$
\underline{\mathcal B}_{(Y,K)}(j_*(h))\ne\emptyset,\qquad\mathcal B_{(Y,K)}(h)\ne\emptyset.
$$
\end{prop}
\begin{proof}
Let $F$ be a rational Seifert surface for $K$, then $[F]+mj_*(h)$ is represented by a rational Seifert-like surface for $K$ for any $m\in\mathbb Z$. The same argument as in Lemma~\ref{lem:h1h2} shows that $$\underline{\mathcal B}_{(Y,K)}([F]+mj_*(h))\subset\underline{\mathcal B}_{(Y,K)}(j_*(h)),\quad\text{when $m$ is sufficiently large.}$$
Proposition~\ref{prop:ExtremeSpinc} implies that $\underline{\mathcal B}_{(Y,K)}([F]+mj_*(h))\ne\emptyset$, so $\underline{\mathcal B}_{(Y,K)}(j_*(h))\ne\emptyset$ and hence $\mathcal B_{(Y,K)}(h)\ne\emptyset$.
\end{proof}

Let $G\subset Y$ be a taut surface. Following Kronheimer and Mrowka \cite{KMSuture}, let 
$$HF^{\circ}(Y|G)=\bigoplus_{c_1(\mathfrak s)\in\mathcal B_Y([G])}HF^{\circ}(Y,\mathfrak s),$$
where $HF^{\circ}$ is one of the ``hat'' and ``$+$'' theories. Moreover, let 
$$HF^{\circ}(Y,[G],i)=\bigoplus_{\mathfrak s\in\spinc(Y),
\langle c_1(\mathfrak s),[G]\rangle=2i}HF^{\circ}(Y,\mathfrak s).$$

$HF^{\circ}(Y|G)$ contains information about fibrations on $Y$. Let us recall the main theorem in \cite{NiClosedFib}. (See also \cite{Gh,NiFibred,AiNi}.)

\begin{thm}\label{thm:ClosedFibre}
Suppose that $Y$ is a closed irreducible $3$--manifold, and $G\subset Y$ is
a connected surface of genus $g\ge2$.  
If $HF^+(Y|G)\cong\mathbb Z$, then $Y$ fibers over the
circle with $G$ as a fiber.
\end{thm}

The following proposition indicates that $HF^{\circ}(Y|G)$ is actually an invariant for the corresponding sutured manifolds.

\begin{prop}\label{prop:CutReglue}
Suppose $G\subset Y$ is a closed connected surface with $g(G)>1$, and $G\cap K=\emptyset$. Let $Y'$ be a manifold obtained from $Y$ by cutting $Y$ open along $G$ then gluing via a homeomorphism of $G$, then $K$ becomes a knot $K'$ in $Y'$. Then for any $i\le1-g(G)$ we have isomorphisms:
$${HF}^{\circ}(Y,[G],i)\cong{HF}^{\circ}(Y',[G],i),\quad {HFK}^{\circ}(Y,K,[G],i)\cong{HFK}^{\circ}(Y',K',[G],i).$$
\end{prop}

The proof of this proposition is standard. It either follows from the surgery exact triangle as in Ni \cite[Proposition~3.5]{NiSuturedD} or a version of the excision formula as in Kronheimer--Mrowka \cite{KMSuture}.

\subsection{Applying the internal hierarchy}

As an immediate application of Theorem~\ref{thm:PreciseGabai}, we prove Theorem~\ref{thm:ThNorm} in the case $\chi_-(h)>0$.

\begin{prop}\label{prop:NormSummand}
Suppose that $Y$ is a closed oriented $3$--manifold, and that $G\subset Y$ is a taut surface with $\chi(G)<0$. Then $\widehat{HF}(Y|G)$ contains a $\mathbb Z\oplus\mathbb Z$ direct summand, and $HF^+(Y|G)$ contains a $\mathbb Z$ direct summand.
\end{prop}
\begin{proof}
Using the K\"unneth formula for connected sums if necessary, we can reduce our problem to the case that $Y$ is irreducible.

We first deal with the case when $G$ is connected.
By Theorem~\ref{thm:PreciseGabai}, there exits a sequence of pairs $$(Y,G)=(Y_0,G_0),(Y_1,G_1),\dots,(Y_n,G_n)=(G_n\times S^1,G_n),$$
such that $Y_{i+1}$ is obtained from $Y_i$ by cutting open $Y_i$ along $G_i$ and regluing via a homeomorphism of $G_i$, and
(\ref{eq:SpincSub}) holds.
By Proposition~\ref{prop:CutReglue}, $${HF}^{\circ}(Y_{i+1}|G_i)\cong{HF}^{\circ}(Y_i|G_i).$$
So ${HF}^{\circ}(Y_{i+1}|G_{i+1})$ is a direct summand of ${HF}^{\circ}(Y_i|G_i)$. As a consequence, ${HF}^{\circ}(G_n\times S^1|G_n)$ is a direct summand of $HF^{\circ}(Y|G)$.

The proof of the general case is sketched as follows. Suppose that $E_0$ is a component of $G=G_0$ with $g(E_0)>1$. Let $M_0=Y\Bslash E_0$. Suppose that $S\subset M_0$ is a taut surface satisfying the condition of Remark~\ref{rem:OneCurve}, we can isotope $S$ such that $|S\cap(G-E_0)|$ is as small as possible. Let $(Y_1,\overline S)$ be the pair obtained from $(M_0,S)$ by gluing $R_+(M_0)$ to $R_-(M_0)$ via a homeomorphism sending $\partial_+S$ to $\partial_-S$, then $HF^{\circ}(Y_1|G)\cong HF^{\circ}(Y|G)$.
Using Lemma~\ref{lem:h1h2}, $\mathcal B_{Y_1}(\overline{S}+m[G])\subset\mathcal B_{Y_1}([G])$ when $m$ is sufficiently large. Consider the surface obtained from ${S}$ and $m$ copies of $G-E_0$ by cut-and-pastes, let $S_1$ be the union of its non-closed components\footnote{
By Remark~\ref{rem:OneCurve}, $S$ has no closed components. So the only closed components come from the closed components of $G-E_0$ that do not intersect $S$.} 
and let
$\overline{S_1}$ be the corresponding closed surface in $Y_1$. Let $G_1$ be the surface obtained from $\overline S$ and $m$ copies of $G$ by cut-and-pastes, and let
 $E_1$ be the component of $G_1$ which is obtained from $\overline{S_1}$ and $m$ copies of $E_0$ by cut-and-pastes.
 
Using the argument in the proof of Theorem~\ref{thm:PreciseGabai}, we can get a triple $$(Y_{n_1},G_{n_1},E_{n_1}),$$ such that $G_{n_1}$ is taut, $HF^{\circ}(Y_{n_1}|G_{n_1})$ is a direct summand of $HF^{\circ}(Y|G)$, $E_{n_1}$ is a component of $G_{n_1}$ with genus $>1$, and $C(Y_{n_1},E_{n_1})<C(Y,E_0)$. So we can repeat this process until we get a triple $(Y_{n},G_{n},E_{n})$ such that $Y_n$ fibers over $S^1$ with fiber $E_n$. In this case $G_n$ must contain parallel copies of $E_n$. Now our conclusion holds.
\end{proof}

The argument above tells us slightly more than just the nontriviality of $HF^{\circ}(Y|G)$. Recall that for any $\zeta\in H_1(Y)/\mathrm{Tors}$ there is a map $A_{\zeta}\co HF^{\circ}(Y,\mathfrak s)\to HF^{\circ}(Y,\mathfrak s)$ satisfying $A^{2}_{\zeta}=0$ \cite{OSzAnn1}. Thus $A_{\zeta}$ can be regarded as a differential on $HF^{\circ}(Y,\mathfrak s)$. It is not hard to see that $A_{\zeta}$ respects the isomorphism in Proposition~\ref{prop:CutReglue}. (More precisely, let $\omega$ be a $1$--cycle representing $\zeta$. We can realize the cut-and-reglue process in Proposition~\ref{prop:CutReglue} by Dehn surgery on a link $L$ contained in $G$, and $\omega$ can be chosen to be disjoint from $L$. Hence $\omega$ corresponds to a $1$--cycle $\omega'$ in $Y'$. Then $A_{[\omega]}$ and $A_{[\omega']}$ coincide under the isomorphism in Proposition~\ref{prop:CutReglue}.) We also observe that any $A_{\zeta}$ map on $HF^{\circ}(G_n\times S^1|G_n)$ is zero. So the argument in Proposition~\ref{prop:NormSummand} implies the following theorem.

\begin{thm}\label{thm:NormAct}
Suppose that $Y$ is a closed oriented $3$--manifold, and that $G\subset Y$ is a taut surface with $\chi(G)<0$. Then for any $\zeta\in H_1(Y)/\mathrm{Tors}$, the homology $H_*(HF^{\circ}(Y|G),A_{\zeta})$ has rank greater than or equal to $1$ or $2$, according to whether $HF^{\circ}$ is $HF^+$ or $\widehat{HF}$. Moreover, the group
$$\bigcap_{\zeta\in H_1(Y)/\mathrm{Tors}}\ker \big(A_{\zeta}\co HF^{\circ}(Y|G)\to HF^{\circ}(Y|G)\big)$$
has rank greater than or equal to $1$ or $2$.
\end{thm}

This approach can also be used to prove Theorem~\ref{thm:ClosedFibre}. For simplicity, we do not give the proof here. Instead, assuming Theorem~\ref{thm:ClosedFibre}, we will show how to use the internal hierarchy to refine the theorem in the sense of considering the homological action.
See also \cite{NiBorro} for the version for sutured Floer homology.

\begin{thm}
Suppose that $Y$ is a closed irreducible oriented $3$--manifold, and that $G\subset Y$ is a taut connected surface with $\chi(G)<0$. Assume that $Y$ does not fiber over $S^1$ with fiber $G$. Then for any $\zeta\in H_1(Y)/\mathrm{Tors}$, the homology $H_*(HF^+(Y|G),A_{\zeta})$ has rank greater than $1$. Moreover, the group
$$\bigcap_{\zeta\in H_1(Y)/\mathrm{Tors}}\ker \big(A_{\zeta}\co HF^+(Y|G)\to HF^+(Y|G)\big)$$
has rank greater than $1$.
\end{thm}
\begin{proof}
Let $S\subset M$ be as in Remark~\ref{rem:OneCurve}, let $Y',\overline S,G'$ be as in Construction~\ref{constr:Reglue}. We only need to prove the theorem for $HF^+(Y'|G)$. By Theorem~\ref{thm:PreciseGabai}, we can choose $S$ so that $$\mathcal B_{Y'}([\overline S]+[G])=\mathcal B_{Y'}([G'])\subset \mathcal B_{Y'}([G]).$$ By Lemma~\ref{lem:h1h2}, $$\mathcal B_{Y'}(-[\overline S]+m[G])\subset \mathcal B_{Y'}([G])$$
when $m$ is sufficiently large. 

If 
$$\chi_-([\overline S]+[G])+\chi_-(-[\overline S]+m[G])>\chi_-((m+1)[G]),$$
then by Lemma~\ref{lem:h1h2} $$\mathcal B_{Y'}([G'])\cap\mathcal B_{Y'}(-[\overline S]+m[G])\cap \mathcal B_{Y'}((m+1)[G])=\emptyset.$$
Since $$\mathcal B_{Y'}(-[\overline S]+m[G])\subset \mathcal B_{Y'}([G])=\mathcal B_{Y'}((m+1)[G]),$$ we have
$$\mathcal B_{Y'}([G'])\cap\mathcal B_{Y'}(-[\overline S]+m[G])=\emptyset.$$
Applying Theorem~\ref{thm:NormAct} to $G'$ and a taut surface representing $-[\overline S]+m[G]$, we get our conclusion. 

If $$\chi_-([\overline S]+[G])+\chi_-(-[\overline S]+m[G])=\chi_-((m+1)[G]),$$
then by Lemma~\ref{lem:h1h2} $$\mathcal B_{Y'}([G'])\cap\mathcal B_{Y'}(-[\overline S]+m[G])=\mathcal B_{Y'}((m+1)[G])=\mathcal B_{Y'}([G]).$$
It follows that $\mathcal B_{Y'}([G'])=\mathcal B_{Y'}([G])$. Thus $HF^+(Y'|G')\cong HF^+(Y'|G)$. 

Applying Theorem~\ref{thm:PreciseGabai}, either our conclusion holds, or we get a pair $(G_n\times S^1,G_n)\prec(Y,G)$ with $$HF^+(G_n\times S^1|G_n)\cong HF^+(Y|G).$$ By Theorem~\ref{thm:ClosedFibre}, $Y$ fibers over $S^1$ with fiber $G$, a contradiction.
\end{proof}


\section{Floer simple knots}\label{Sect:FSimple}

In this section, we will prove Theorem~\ref{thm:Nullhomotopic}.
If $K$ is contained in a $3$--ball then the desired result holds by \cite{OSzKnot}. From now on, we assume $K$ is not contained in a ball, then $X=Y\Bslash K$ is irreducible.

\begin{defn}
Suppose $K$ is a rationally null-homologous knot in $Y$, $h\in H_2(Y)$ is a homology class. We say $K$ is {\it bottommostly Floer simple relative to $h$} if 
$$\mathrm{rank}\bigoplus_{\langle c_1(\mathfrak s),h\rangle\le -\chi_-(h)}\widehat{HFK}(Y,K,\mathfrak s)=\mathrm{rank}\bigoplus_{\langle c_1(\mathfrak s),h\rangle\le -\chi_-(h)}\widehat{HF}(Y,\mathfrak s).$$ In other words, $K$ is bottommostly Floer simple relative to $h$ if
$$\mathrm{rank}\widehat{HFK}(Y,K,\mathfrak s)=\mathrm{rank}\widehat{HF}(Y,\mathfrak s)$$
for any $\mathfrak s\in\spinc(Y)$ with $\langle c_1(\mathfrak s),h\rangle\le-\chi_-(h)$.
\end{defn}

\begin{lem}\label{lem:DisjointTaut}
Suppose that $K\subset Y$ is a rationally null-homologous knot, and that $K$  is bottommostly Floer simple relative to $h\in H_2(Y)$. Then the Thurston norm of $h$ in $Y$ is equal to its Thurston norm in $X$. 
\end{lem}
\begin{proof}
As $K$ is rationally null-homologous,  $$\mathrm{rank}\:\widehat{HFK}(Y,K,\mathfrak s)\ge\mathrm{rank}\:\widehat{HF}(Y,\mathfrak s)$$
for any $\mathfrak s\in\spinc(Y)$. The assumption that $K$  is bottommostly simple relative to $h$ then implies that the above equality holds for any $\mathfrak s$ with $\langle c_1(\mathfrak s),h\rangle\le-\chi_-(h)$.
Using Theorem~\ref{thm:ThNorm} and Proposition~\ref{prop:ThNormK}, we get our conclusion.
\end{proof}

\begin{prop}\label{prop:DecrCWKnot}
Suppose that $G$ is a connected taut surface in $Y$ with $g(G)>1$, and that $Y$ does not fiber over $S^1$ with fiber $G$. Suppose that $K\subset Y$ is a null-homotopic knot which is not contained in a $3$--ball, $K\cap G=\emptyset$ and $K$ is bottommostly simple relative to $[G]$. Then we can find a sequence of triples 
$$(Y,K,G)=(Y_0,K_0,G_0), (Y_1,K_1,G_1), \dots, (Y_n,K_n,G_n),$$
such that the following conditions hold:

(a) $K_i\cap G_i=\emptyset$, and $(Y_{i+1},K_{i+1})$ is obtained from $(Y_i,K_i)$ by cutting open $Y_i$ along $G_i$ and regluing by a homeomorphism of $G_i$;

(b) $G_i$ is connected and taut in both $Y_i$ and $Y_i\Bslash K_i$;

(c) $K_i$ is bottommostly simple relative to $[G_i]$;

(d) $C(Y_n,G_n)<C(Y,G)$.
\end{prop}

We modify the proof of Theorem~\ref{thm:PreciseGabai}.

As $G$ is incompressible, $K$ is null-homotopic in $M=Y\Bslash G$. So if $S\subset M$ is a properly embedded surface, we can always add tubes to $S$ to get a surface $S'\subset M$ with $\partial S'=\partial S$ and $S'\cap K=\emptyset$.

Let $E=M\Bslash K$, $\rho=\partial\nu(K)$, then $(E,\rho)$ is naturally a sutured manifold. We will apply Construction~\ref{constr:Reglue} repeatedly, but with the difference that $S$ is chosen to be a taut surface in $(E,\rho)$ with $S\cap\rho=\emptyset$. Recall from Construction~\ref{constr:Reglue} that $(Y',\overline S)$ is obtained from $(M,S)$ by gluing $G_+$ to $G_-$. Let $K'$ be the new knot in $Y'$, then $K'$ is null-homotopic in $Y'$. Let $G^{(m)}$ be the surface obtained from $\overline S$ and $m$ copies of $G$ by 
cut-and-pastes, then $G^{(m)}$ is taut in $Y'\Bslash K'$ since $S$ is taut in $(E,\rho)$.

\begin{lem}\label{lem:LowerSub}
Let $\mathcal B$ be a finite subset of $H^2(Y')$. When $m$ is sufficiently large, we have 
$$\big\{\alpha\in\mathcal B|\:\langle \alpha,[G^{(m)}]\rangle\le\chi(G^{(m)})\big\}\subset \big\{\alpha\in\mathcal B\big|\:\langle \alpha,[G]\rangle\le\chi(G)\big\}.$$
\end{lem}
\begin{proof}
Let $m$ be an integer greater than $|\langle \alpha,[\overline S]\rangle-\chi(\overline S)|$ for all $\alpha\in\mathcal B$. If $\alpha\in\mathcal B$ satisfies that
$$\langle \alpha,[G^{(m)}]\rangle\le\chi(G^{(m)})\quad\text{and}\quad\langle \alpha,[G]\rangle>\chi(G),$$
then we have
\begin{eqnarray*}
\langle\alpha,[\overline S]\rangle+m\chi(G)
&\le&
\langle \alpha,[\overline S]\rangle+m(\langle \alpha,[G]\rangle-1)\\
&=&\langle \alpha,[G^{(m)}]\rangle-m\\
&\le&\chi(G^{(m)})-m\\
&=&\chi(\overline S)+m\chi(G)-m,
\end{eqnarray*}
which contradicts the choice of $m$.
\end{proof}

Using Proposition~\ref{prop:CutReglue}, $K'$ is also bottommostly Floer simple relative to $[G]$ in $Y'$. 
Applying Lemma~\ref{lem:LowerSub} to $\mathcal B=\mathcal B_{(Y',K')}$, we conclude that 
\begin{equation}\label{eq:Y'Simple}
\begin{array}{c}
\mathrm{rank}\widehat{HFK}(Y',K',\mathfrak s)=\mathrm{rank}\widehat{HF}(Y',\mathfrak s),\\
\text{for any } \mathfrak s\in\spinc(Y') \text{ with }\langle c_1(\mathfrak s),[G^{(m)}]\rangle\le\chi(G^{(m)}).
\end{array}
\end{equation}
Since $G^{(m)}$ is taut in $Y'\Bslash K'$, by Proposition~\ref{prop:ThNormK} $\mathrm{rank}\widehat{HFK}(Y',K',\mathfrak s)\ne0$ for some $\mathfrak s$ with $\langle c_1(\mathfrak s),[G^{(m)}]\rangle=\chi(G^{(m)})$. So 
$\mathrm{rank}\widehat{HF}(Y',\mathfrak s)\ne0$, which implies that $G^{(m)}$ is taut in $Y'$.
Now (\ref{eq:Y'Simple}) means that $K'$ is also bottommostly simple relative to $[G^{(m)}]$ in $Y'$, when $m$ sufficiently large.

Let $S^{(m-1)}\subset M$ be the surface obtained from $S$ and $(m-1)$ copies of $G_+$ by cut-and-pastes. We can replace $S$ with $S^{(m-1)}$, thus replace $G'$ with $G^{(m)}$. As we have showed, $G^{(m)}$ is taut in $Y'$, hence $S^{(m-1)}$ is taut in $M$. 

With the above understanding, the proof of Proposition~\ref{prop:DecrCWKnot} proceeds exactly as in the proof of Theorem~\ref{thm:PreciseGabai}. We omit the detail and leave it to the reader.

\begin{proof}[Proof of Theorem~\ref{thm:Nullhomotopic}]
Assume $K$ is nontrivial, then $K$ is not contained in a $3$--ball and $Y\Bslash K$ is irreducible. Let $h\in H_2(Y)$ be a homology class with $\chi_-(h)\ne0$. Using Lemma~\ref{lem:DisjointTaut}, we can find a taut surface $G$ representing $h$ such that $K$ is disjoint from $G$. We may assume $G$ is connected, otherwise we work with a nontorus component of $G$ instead.

Applying Proposition~\ref{prop:DecrCWKnot} repeatedly, we get a sequence of triples
$$(Y,K,G)=(Y_0,K_0,G_0), (Y_1,K_1,G_1), \dots, (Y_n,K_n,G_n),$$
such that (a), (b), (c) hold and $(Y_n,G_n)=(G_n\times S^1,G_n)$. Since $Y_n\Bslash K_n$ is obtained from the irreducible manifold $Y\Bslash K$ by cutting and regluing along incompressible surfaces, $Y_n\Bslash K_n$ is also irreducible, so $K_n$ is nontrivial. We get a contradiction by using Theorem~\ref{thm:BundBottom} from the following subsection.
\end{proof}

\subsection{Rationally null-homologous knots in surface bundles over $S^1$}

Suppose that $Y$ is a surface bundle over $S^1$, with fiber $G$ a closed oriented surface of genus $>1$.  Let $\mathfrak s_{\min}$ be the Spin$^c$ structure over $Y$ such that $\langle c_1(\mathfrak s_{\min}),[G]\rangle=2-2g(G)$, and $HF^+(Y,\mathfrak s_{\min})\cong \mathbb Z$. 
Suppose $K\subset Y$ is a rationally null-homologous knot which is bottommostly Floer simple relative to $[G]$, then $\widehat{HFK}(Y,K,\mathfrak s_{\min})\cong\mathbb Z^2$. 

Since $Y$ fibers over $S^1$, any taut surface in the fiber class $[G]$ must be isotopic to $G$ in $Y$. So Lemma~\ref{lem:DisjointTaut} implies that $K$ can be isotoped to be disjoint from $G$. Cutting $Y$ open along $G$ and regluing via a suitable homeomorphism, we get a manifold $Y_1$ with $H_1(Y_1)\cong\mathbb Z$. Since $K$ is disjoint from $G$, $K$ is null-homologous in $Y_1$. Working with $Y_1$ instead of $Y$ if necessary, we may assume that $K$ is null-homologous in $Y$.

Consider the chain complex $\big(C=CFK^{\infty}(Y,K,\mathfrak s_{\min}),\partial^{\infty}\big)$. There is a map on $C$:
$$U[\mathbf x,i,j]=[\mathbf x,i-1,j-1].$$
The number $i+j$ gives a filtration on $C$. Let $\partial_0$ be the component of $\partial^{\infty}$ which preserves the filtration, then $\partial^{\infty}=\partial_0+\partial_{>0}$. 

Let $\widetilde H=H_*(C,\partial_0)$, then $$\widetilde H(i=0)\cong\widehat{HFK}(Y,K,\mathfrak s_{\min})\cong \mathbb Z^2$$ is generated by two elements $\tensor x,\tensor y$ with different absolute $\mathbb Z/2\mathbb Z$ gradings. Moreover, as a free $\mathbb Z[U,U^{-1}]$--module, $\widetilde H$ is generated by $\tensor x,\tensor y$.

The map $\partial_{>0}$ induces a differential on $\widetilde H$, denoted $\widetilde{\partial}_{>0}$. The homology of $\widetilde H$ with respect to $\widetilde{\partial}_{>0}$ is isomorphic to $HFK^{\infty}(Y,K,\mathfrak s_{\min})$. Since $\tensor x,\tensor y$ have different absolute $\mathbb Z/2\mathbb Z$ gradings, we must have 
$$\widetilde{\partial}_{>0}\tensor x=f(U)\tensor y,\quad \widetilde{\partial}_{>0}\tensor y=g(U)\tensor x$$
for some polynomials $f(U),g(U)\in\mathbb Z[U]$. Since $\widetilde{\partial}_{>0}^2=0$, one of $f,g$ must be zero. Without loss of generality, we may assume $g(U)=0$. 

As $\widehat{HF}(Y,\mathfrak s_{\min})=\mathbb Z^2$, we have $f(0)=0$.
As $HF^+(Y,\mathfrak s_{\min})\cong H_*(\widetilde H(i\ge0),\widetilde{\partial}_{>0})\cong\mathbb Z$, we have 
\begin{equation}\label{eq:U}
f(U)=\pm U+\text{higher order terms}.
\end{equation}

\begin{lem}\label{lem:UniqMin}
The group $\widehat{HFK}(Y,K,\mathfrak s_{\min})$ is supported in a unique relative Spin$^c$ structure.
\end{lem}
\begin{proof}
Since $K$ is null-homologous, $HF^+(Y,\mathfrak s_{\min})\cong\mathbb Z$ can also be computed as $H_*(\widetilde H(j\ge0),\widetilde{\partial}_{>0})$. Suppose $\widetilde H(j=0)\cong\mathbb Z^2$ is generated by $U^a\tensor x, U^b\tensor y$ for $a,b\in\mathbb Z$. Since $H_*(\widetilde H(j\ge0),\widetilde{\partial}_{>0})\cong\mathbb Z$, it follows from (\ref{eq:U}) that $a=b$.

Since $\widetilde{\partial}_{>0}\tensor x$ involves $\tensor y$, there is a holomorphic disk $\phi$ connecting $\mathbf x_1$ to $\mathbf y_1$ for some $\mathbf x_1,\mathbf y_1\in\mathbb T_{\alpha}\cap\mathbb T_{\beta}$ such that $[\mathbf x_1,0,a]$ is a summand in a representative of $\tensor x$ and $[\mathbf y_1,0,a]$ is a summand in a representative of $\tensor y$.
By Equation (\ref{eq:U}), $n_z(\phi)=n_w(\phi)=1$. Then $\phi-\Sigma$ is a topological disk connecting $\mathbf x_1$ to $\mathbf y_1$ with $n_z=n_w=0$.
It follows that $\mathbf x_1$ and $\mathbf y_1$ represent the same relative Spin$^c$ structure, and so do $\tensor x$ and $\tensor y$.
\end{proof}

Theorem~\ref{thm:BundleCase} is implied by the next theorem.

\begin{thm}\label{thm:BundBottom}
Suppose that $Y$ is a surface bundle over $S^1$, with fiber $G$ a closed oriented surface of genus $>1$. Let $K\subset Y$ be a rationally null-homologous knot.
If $K$ is bottommostly Floer simple relative to $[G]$, then $K$ is the unknot.
\end{thm}
\begin{proof}
Without loss of generality, we may assume $K$ is null-homologous. It follows from Lemma~\ref{lem:UniqMin} that $c_1(\xi_{\min})$ is the unique bottommost basic class relative to $[G]$. Let $F$ be a minimal genus Seifert surface for $K$. As argued in Lemma~\ref{lem:LowerSub}, $c_1(\xi_{\min})$ is also the unique bottommost class for $n[G]\pm[F]$ when $n\in \mathbb Z$ is sufficiently large. Assume that $K$ is not the unknot, we have
\begin{eqnarray*}
-2n\chi_-(G)&=&\langle c_1(\xi_{\min}),2n[G]\rangle\\
&=&\langle c_1(\xi_{\min}),n[G]+[F]\rangle+\langle c_1(\xi_{\min}),n[G]-[F]\rangle\\
&=&-\chi_-(n[G]+[F])+(-\chi_-(n[G]-[F])-2)\\
&\le&-\chi_-(2n[G])-2\\
&=&-2n\chi_-(G)-2,
\end{eqnarray*}
a contradiction. The third equality in the above computation uses Proposition~\ref{prop:ExtremeSpinc}. We note that $-(n[G]-[F])$ is the homology class of a Seifert-like surface for $K$, and $c_1(\xi_{\min})$ is the unique topmost class relative to $-(n[G]-[F])$. So $\langle c_1(\xi_{\min}),-(n[G]-[F])\rangle=\chi_-(-(n[G]-[F]))+2$ by the second equality in Proposition~\ref{prop:ExtremeSpinc}.
\end{proof}

In the statement of Theorem~\ref{thm:BundleCase}, we require that the genus of the fiber $G$ is greater than $1$. If $g(G)=1$, we can get the same conclusion in some cases. For example, if $Y=T^3$, then $\widehat{HF}(Y)$ is supported in two absolute gradings $+\frac12,-\frac12$ \cite[Proposition~8.4]{OSzAbGr}. If $K\subset T^3$ is a Floer simple knot, then $\widehat{HFK}(Y,K)$ is also supported in these two absolute gradings. Using \cite[Proposition~3.10]{OSzKnot}, we see that the genus of $K$ must be zero, so $K$ is the unknot. Another way to see this is to use Lemma~\ref{lem:DisjointTaut}, which implies that any essential torus in $Y$ can be isotoped to be disjoint from $K$. So $K$ lies in a $3$--ball. Then we can use \cite{OSzGenus} to conclude that $K$ is the unknot.

\end{document}